\documentclass[12pt]{amsart}

\usepackage{amssymb}
\usepackage{esint}
\usepackage{mathtools}
\usepackage[english]{babel}
\usepackage{epsfig}
\usepackage{mathptmx}
\usepackage{amsmath,amsfonts,amssymb}
\usepackage{mathtools}
\usepackage{enumitem}
\usepackage{mathrsfs}
\usepackage[all]{xy}
\usepackage{graphicx}
\usepackage{latexsym}
\usepackage{verbatim}
\usepackage{xcolor}
\usepackage{ifthen}
\usepackage{cancel}
\usepackage[normalem]{ulem}
\usepackage[hidelinks]{hyperref}

\usepackage[text={16.5cm, 20cm}, asymmetric]{geometry}
\numberwithin{equation}{section}

\def\Re{\mathop{\mathrm{Re}}}
\def\Im{\mathop{\mathrm{Im}}}

\def\D{\mathbb D}

\def\H{\mathbb{H}}

\def\R{\mathbb R}
\def\Z{\mathbb Z}
\def\C{\mathbb C}
\def\Hol{{\sf Hol}}
\def\Aut{{\sf Aut}(\mathbb D)}

\def\N{\mathbb N}
\def\Q{\mathbb Q}

\def\id{{\sf id}}
\def\Arg{\mathop{\mathrm{Arg}}}

\newtheorem{theorem}{Theorem}[section]
\newtheorem{lemma}[theorem]{Lemma}
\newtheorem{proposition}[theorem]{Proposition}
\newtheorem{corollary}[theorem]{Corollary}

\theoremstyle{definition}
\newtheorem{definition}[theorem]{Definition}
\newtheorem{example}[theorem]{Example}

\theoremstyle{remark}
\newtheorem{remark}[theorem]{Remark}
\newtheorem{problem}[theorem]{Problem}

\numberwithin{equation}{section}

\newcommand{\Zen}{\mathcal{Z}}

\newcommand{\U}{\mathcal{U}}
\newcommand{\UH}{\mathbb{H}}
\newcommand{\Hr}{\mathbb{H}_{\mathrm{Re}}}
\newcommand{\UD}{\mathbb{D}}
\newcommand{\UC}{\partial\UD}
\newcommand{\Complex}{\mathbb{C}}
\newcommand{\Real}{\mathbb{R}}
\newcommand{\Natural}{\mathbb{N}}
\newcommand{\di}{\mathop{\mathrm{d}}\nolimits}
\newcommand{\anglim}{\angle\lim}

\renewcommand{\emptyset}{\varnothing}
\renewcommand{\ge}{\geqslant}
\renewcommand{\le}{\leqslant}
\renewcommand{\geq}{\geqslant}
\renewcommand{\leq}{\leqslant}
\newcommand{\proofof}[1]{{\fontseries{bx}\fontshape{it}\selectfont Proof of #1}}

\newcommand{\StepP}[1]{\medskip\noindent{\textsc{Proof of #1.}}}

\newcommand{\StepC}[2]{\medskip\noindent{\textsc{Case}~#1: #2.}}
\newcommand{\StepG}[1]{\medskip\noindent{\textsc{#1}}}

\newenvironment{ourlist}{\begin{enumerate}[label={\rm (\arabic*)}, ref={\rm (\arabic*)}, left=.7em]}{\end{enumerate}}
\newenvironment{romlist}{\begin{enumerate}[label={\rm (\roman*)}, ref={\rm (\roman*)}, left=.7em]}{\end{enumerate}}
\newenvironment{Romlist}{\begin{enumerate}[label={\bf (\Roman*)}, ref={\hbox{\rm (\hskip.07em\Roman*\hskip.07em)}}, left=.7em]}{\end{enumerate}}
\newenvironment{equilist}{\begin{enumerate}[label={\rm (\alph*)}, ref={\rm (\alph*)}, left=.7em]}{\end{enumerate}}
\newenvironment{statlist}{\begin{enumerate}[label={\bf (\Alph*)}, ref={\rm (\Alph*)}, left=0em]%
\everydisplay{\makeatletter\def\@eqnum{\normalfont(\theequation)}\makeatother}}{\end{enumerate}}

\newcommand{\dff}[1]{\textsl{#1}}

\begin{document}
\title[Extension of commutativity to fractional iterates]{Criteria for extension of commutativity to fractional iterates of holomorphic self-maps in the unit disc}

\author[M. D. Contreras]{Manuel D. Contreras $^\dag$}

\author[S. D\'{\i}az-Madrigal]{Santiago D\'{\i}az-Madrigal $^\dag$}
\address{Camino de los Descubrimientos, s/n\\
	Departamento de Matem\'{a}tica Aplicada II and IMUS\\
	Universidad de Sevilla\\
	Sevilla, 41092\\
	Spain.}\email{contreras@us.es} \email{madrigal@us.es}

\author[P. Gumenyuk]{Pavel Gumenyuk$^\ddag$}\address{Pavel Gumenyuk: Department of
Mathematics\\ Milano Politecnico, via E. Bonardi 9\\ Milan 20133, Italy.}
\email{pavel.gumenyuk@polimi.it}

\date\today

\thanks{This research was supported in part by Ministerio de Innovaci\'on y Ciencia, Spain, project PID2022-136320NB-I00.}

\begin{abstract} Let $\varphi$ be a univalent non-elliptic self-map of the unit disc~$\UD$ and let $(\psi_{t})$ be a continuous one-parameter semigroup of holomorphic functions in $\UD$ such that $\psi_{1}\neq\id_\UD$ commutes with~$\varphi$. This assumption does not imply that all elements of the semigroup~$(\psi_t)$ commute with~$\varphi$. In this paper, we provide a number of sufficient conditions that guarantee that ${\psi_t\circ\varphi=\varphi\circ\psi_t}$ for all ${t>0}$: this holds, for example, if $\varphi$ and $\psi_1$ have a common boundary (regular or irregular) fixed point different from their common Denjoy\,--\,Wolff point~$\tau$, or when $\psi_1$ has a boundary regular fixed point ${\sigma\neq\tau}$ at which $\varphi$ is isogonal, or when $(\varphi-\id_\UD)/(\psi_1-\id_\UD)$ has an unrestricted limit at~$\tau$. In addition, we analyze how $\varphi$ behaves in the petals of the semigroup~$(\psi_t)$.
\end{abstract}

\maketitle

\tableofcontents

\section{Introduction and main results}
This paper is motivated by the following natural and quite old problem in discrete holomorphic iteration: given a holomorphic self-map of the unit disc, i.e. $\varphi\in\Hol(\D)$, determine or at least analyse those $\psi\in\Hol(\D)$ that commute with $\varphi$, see \cite[Section~4.10]{Abate2}. This kind of questions has also been treated in the framework of fractional iteration of holomorphic self-maps, with the aim to study all continuous one-parameter semigroups in the unit disc $(\psi_t)$ that commute with a given continuous one-parameter semigroup $(\varphi_t)$ in the sense that $\psi_t\circ\varphi_s=\varphi_s\circ\psi_t$ for all $t,s\geq0$, see for instance~\cite{conReich} and~\cite{conTauraso}.

In this paper, we tackle an intermediate situation, which has interesting implications in the discrete as well as in the fractional framework. Given $\varphi\in\U(\D)$, a univalent self-map of the unit disc, the \dff{centralizer}~$\Zen(\varphi)$ of~$\varphi$ is defined as
$$
 \Zen(\varphi):=\{\psi\in\U(\D):\varphi\circ\psi=\psi\circ\varphi\}, \quad \U(\UD):=\{\psi:\UD\to\UD~\text{~holomorphic injective}\}.
$$
The problem we are interested in is the following:
\begin{problem}\label{theproblem}
Fix $\varphi\in \U(\D)$ and a continuous one-parameter semigroup in the unit disc $(\psi_{t})$. Suppose that ${\psi_1\in\Zen(\varphi)}$, with $\psi_1\neq\id_\UD$. Does it necessarily follow that $\psi_{t}\in \Zen(\varphi)$ \textit{for all}~${t>0}$?  If not, provide  conditions on $\varphi$ and $\psi_1$~--- and maybe on some finite number of $\psi_t$'s~--- under which the relation ${\psi_1\in\Zen(\varphi)}$ implies that  the whole semigroup $(\psi_t)$ is contained in~$\Zen(\varphi)$.
\end{problem}
First of all,
it is worth mentioning that the assumption of $\varphi$ being univalent is completely natural and not restrictive. As we will see (Propositions~\ref{PR_simple} and~\ref{PR_elliptic-case}), univalence is a necessary condition for $\varphi$ to commute with a non-trivial continuous one-parameter semigroup.
We would also like to underline that the context of this problem is really different from the other two mentioned above because of its lack of symmetry.
In particular, for two continuous one-parameter semigroups in the unit disc, $(\varphi_t)$ and $(\psi_t)$, we can have ${\psi_1\in\Zen(\varphi_t)}$ for all ${t>0}$ and at the same time, ${\varphi_1\not\in\Zen(\psi_t)}$ for some~${t>0}$; see Example~\ref{Mainexample}. The same example shows that, in general, the answer to the question in Problem~\ref{theproblem} is negative.

When $\varphi$ is elliptic, the above Problem~\ref{theproblem} was implicitly solved by Cowen~\cite{Cowen-comm}. For the sake of completeness,  at the end of the paper we include a brief appendix with a very simple, and independent from Cowen's work,  solution for the elliptic case; see Proposition~\ref{PR_elliptic-case}.
In the rest of the paper, we restrict ourselves to the non-elliptic case.

In that non-elliptic case, it is not difficult to give a complete answer to Problem~\ref{theproblem} in terms of the Koenigs function~$H$ of the semigroup~$(\psi_t)$. Indeed, ${(\psi_t)\subset\Zen(\varphi)}$ if and only there exists ${c\in\C}$ such that ${H\circ \varphi}={H+c}$, see Proposition~\ref{PR_simple}. The latter condition means exactly that $\varphi$ is affine with respect to~$\psi_1$ (see Definition~\ref{Def:affine} and, in general, Section~\ref{SS_cummuting} for further details). Moreover, as a quite direct consequence of our results in~\cite{CDG-Centralizer}, we can answer the question in Problem~\ref{theproblem} positively if the semigroup $(\psi_t)$ is hyperbolic or parabolic of zero hyperbolic step. However, when $(\psi_t)$ is parabolic of positive hyperbolic step, giving any significant answer to Problem~\ref{theproblem} that does not involve~$H$ or the infinitesimal generator ${G=1/H'}$ is really not that easy.

In the following theorem, we summarize our positive answers to Problem~\ref{theproblem} including the two ones described in the former paragraph (as items \ref{IT_main:affine} and~\ref{IT_main:hyperbolic-or-parabolic-zero}). For the definition of isogonality at a boundary point involved in~\ref{IT_main:isogonal}, see Section~\ref{PetalosISO}. Other useful definitions can be found in the preliminaries, see Section~\ref{Notation}. Note that in our terminology every non-elliptic self-map is different from the identity map~$\id_\UD$. Furthermore, by a repelling fixed point we mean a boundary regular fixed point different from the Denjoy\,--\,Wolff point.

\begin{theorem}\label{Th_main}
Let $\varphi\in \U(\D)$ be non-elliptic and  $(\psi_{t})$ be a continuous one-parameter semigroup in~$\UD$ such that $\psi_1\neq\id_\UD$ and $\varphi\circ\psi_{1}=\psi_{1}\circ \varphi$. Assume that one of the following conditions holds:
\newcounter{auxcounter}
\begin{equilist}
\item\label{IT_main:affine}%
   $\varphi$ is affine with respect to $\psi_{1}$,
   \smallskip
\item\label{IT_main:hyperbolic-or-parabolic-zero}%
   $\psi_{1}$ is either hyperbolic or parabolic of zero hyperbolic step,
   \smallskip
\item\label{IT_main:two-values-of-t}%
   there exist $r\in (0,+\infty)\setminus \Q$ such that $\psi_{r}\in  \Zen (\varphi)$,
   \smallskip
\item\label{IT_main:unrestr_limit}%
   the limit
   $$
    \lim_{z\to\tau}\,\frac{\varphi(z)-z}{\,\psi_1(z)-z\,},
   $$
   where $\tau$ is the Denjoy\,--\,Wolff point of~$\psi_1$ (and hence also of~$\varphi$),
   exists unrestrictedly in~$\UD$,
   \smallskip
\item\label{IT_main:isogonal}%
   $\psi_1$ has a repelling fixed point at which $\varphi$ is isogonal,
\setcounter{auxcounter}{\arabic{enumi}}
\end{equilist}
\mbox{\,\,\,\,or}
\begin{equilist}\setcounter{enumi}{\arabic{auxcounter}}
\item\label{IT_main:boundary-fixed-pt}%
   $\varphi$ and $\psi_1$ have a common boundary fixed point different from the Denjoy\,--\,Wolff point.
\end{equilist}
\smallskip
\nopagebreak
Then,  $\psi_{t}\in \Zen(\varphi)$ for all $t>0$.
\end{theorem}
In general, conditions \ref{IT_main:hyperbolic-or-parabolic-zero} and \ref{IT_main:unrestr_limit}\,--\,\ref{IT_main:boundary-fixed-pt} in Theorem~\ref{Th_main} are only sufficient ones.
At the same time, it is worth remarking that these conditions involve $\varphi$ and $\psi_1$, but do not depend on the knowledge of other elements of the semigroup $(\psi_t)$, its Koenigs map, or infinitesimal generator. A priori, this is not the case for the \textit{necessary and sufficient} condition~\ref{IT_main:affine}, see Definition~\ref{Def:affine}. However, in many cases our next result allows to bypass this difficulty.
\begin{theorem}\label{TH_affine}
	Let $\varphi\in\U(\UD)$ be parabolic of positive hyperbolic step, and let $\psi\in\Zen(\varphi)\setminus\{\id_\D\}$. Denote by $\tau$ the Denjoy\,--\,Wolff point  of~$\varphi$ and let
	$$
	R_{\varphi,\psi}(z):= \frac{\psi(z)-z}{\varphi(z)-z},\quad z\in\D.
	$$
	Then, the following statements hold.
	\begin{statlist}
		\item \label{IT_affine:limit-existence}
            The sequence $(R_{\varphi,\psi}\circ \varphi^{\circ n})$ converges locally uniformly in $\D$ to some function $f_{\varphi,\psi}\in\Hol(\D,\C)$.
		\item\label{IT_affine:affine-iff-const}
          $\psi$ is affine with respect to $\varphi$ \,\,if and only if\,\, the function $f_{\varphi,\psi}$ is constant in $\D$. In fact, in such a case, $f_{\varphi,\psi}(\zeta)= \angle\lim_{z\to\tau}\big(\psi(z)-z\big)/\big(\varphi(z)-z\big)$ for all ${\zeta\in\UD}$.
		\smallskip%
		\item\label{IT_affine:extension}%
           Let ${(S,h_\varphi,z\mapsto z+1)}$, where $S=\UH:={\{z\colon \Im z>0\}}$ or ${S=-\UH}$, stand for the canonical holomorphic model of~$\varphi$. Then $f_{\varphi,\psi}\circ h_\varphi^{-1}$ extends holomorphically to a map
		   $G:S\to S\cup\R$ with ${G(w+1)=G(w)}$ for all ${w\in S}$.
		\smallskip%
		\item\label{IT_affine:extension-bis}%
           The self-map $g(w) := w + G(w)$ is univalent in $S$, ${g(h_\varphi(\D))\subset h_\varphi(\D)}$, and
		    $$
                \psi\,=\,h_\varphi^{-1}\circ g\circ h_\varphi\,=\,%
                         h_\varphi^{-1}\circ \big(h_\varphi+f_{\varphi,\psi}\big).
		    $$
	\end{statlist}
\medskip
\end{theorem}

Coming back to Theorem~\ref{Th_main}, condition~\ref{IT_main:boundary-fixed-pt} is probably the most interesting and deepest result of the paper. In fact, according to the following theorem, this condition implies the stronger conclusion that $\varphi$ is an element of~$(\psi_t)$. We exclude from the statement hyperbolic automorphisms,
because in this case the result is essentially known: if $\varphi$ is a hyperbolic automorphism, then $(\psi_t)$ is a hyperbolic one-parameter group and all the three conditions in Theorem~\ref{TH_FixP} below hold, except that in condition~\ref{IT_FixP-in-the-semigroup} the words ``for some~${t_0>0}$'' have to be replaced by ``for some~${t_0\in\Real}$''; see e.g. \cite[Section 4.10]{Abate2} (see also \cite{CDG-Centralizer}).

\begin{theorem}\label{TH_FixP}
Let $\varphi\in\U(\UD)$ be a non-elliptic self-map different from a hyperbolic automorphism, and suppose that it has a boundary fixed point~$\sigma$ different from its Denjoy\,--\,Wolff point. Let $(\psi_t)$ be a continuous one-parameter semigroup such that ${\psi_1\in\Zen(\varphi)}\setminus\{\id_\D\}$. Then the following conditions are equivalent:
\begin{equilist}
\item\label{IT_FixP-in-the-semigroup} ${\varphi=\psi_{t_0}}$ for some~${t_0>0}$;
\item\label{IT_FixP_wholeCOPS-contained} $(\psi_t)\subset\Zen(\varphi)$;
\item\label{IT_FixP_common} $\sigma$ is a boundary fixed point also for~$(\psi_t)$.
\end{equilist}
\end{theorem}

\begin{remark}\label{RM_hyperblic-simple}
If at least one of the self-maps $\varphi$ or $\psi_1$~--- and hence both of them\footnote{See \cite[Corollary~4.1]{Cowen-comm}.}~--- are hyperbolic, then all the equivalent conditions \ref{IT_FixP-in-the-semigroup}, \ref{IT_FixP_wholeCOPS-contained}, and \ref{IT_FixP_common} in the above theorem are automatically satisfied. This follows from the fact that by~\cite[Propositions~6.6 and~6.9]{CDG-Centralizer}, in the hyperbolic case, we have ${\Zen(\varphi)=\Zen(\psi_1)=\{\psi_t:t\ge0\}}$.
\end{remark}

\begin{remark}
 Let $\varphi$ and $\psi$ be two commuting holomorphic self-maps of~$\UD$ and suppose that $\varphi$ has a boundary fixed point~$\sigma$.  It is known (see, e.g., \cite{Filippo-puntosTAMS}; see also Remark~\ref{RM_after-the-example}) that in general,  $\sigma$ does not have to be a boundary fixed point for~$\psi$, even if we additionally assume that $\sigma$ is regular, $\varphi$ is univalent, and $\psi$ is embeddable in a continuous one-parameter semigroup. Therefore, the implications \ref{IT_FixP_wholeCOPS-contained}~$\Rightarrow$~\ref{IT_FixP_common} and \ref{IT_FixP_wholeCOPS-contained}~$\Rightarrow$~\ref{IT_FixP-in-the-semigroup} in the theorem stated above can be regarded as an illustration of a neat difference between commutativity with a self-map (and hence with all its natural iterates) and commutativity with all the fractional iterates.
\end{remark}

If a continuous one-parameter semigroup $(\psi_t)$ is contained in the centralizer of a non-elliptic self-map~$\varphi$, then by Theorem~\ref{TH_FixP} every boundary fixed point of~$\varphi$ has to be among  boundary fixed points of~$(\psi_t)$. The converse is not true: we may have a continuous one-parameter semigroup $(\psi_t)\subset\Zen(\varphi)$ with many boundary fixed points, while $\varphi$ has no boundary fixed point other than its Denjoy\,--\,Wolff point; see, e.g., Example~\ref{Mainexample} with $(\varphi_t)$ and $(\psi_t)$ interchanged. This leads us to the following natural question (in which, in fact, we impose some weaker  assumptions).

 \begin{problem}\label{theproblem2}
Let $\varphi\in\U(\UD)$ be a non-elliptic self-map different from a hyperbolic automorphism and let $(\psi_t)$ be a continuous one-parameter semigroup such that ${\psi_1\in\Zen(\varphi)\setminus\{\id_\UD\}}$. Further, suppose that ${\varphi\not\in\{\psi_t:t\ge0\}}$. Is there any relationship between $\varphi$ and the boundary fixed points of $(\psi_t)$ in this case?
\end{problem}

A natural way to attack this problem is through the concept of petal of a continuous one-parameter semigroup, concept intimately linked to the notion of boundary regular fixed point of such a semigroup. We refer the reader to Section~\ref{petals} for more details, or to the monograph \cite[Section 13]{BCD-Book} for a detailed exposition, which includes a number of related results. It is worth mentioning that this theory of petals plays also an important role in the proof of Theorem~\ref{TH_FixP}.

We begin by showing that, as a general fact, $\varphi$ always maps petals into petals and in a rather specific way. The degenerate case of $(\psi_t)$ being a one-parameter group is excluded from the statement of the theorem below; it is covered separately, see Remark~\ref{RM_group} in Section~\ref{S_petalsII}.

\begin{theorem}\label{petalEST} Let $\varphi\in\U(\D)$ be non-elliptic and let $(\psi_t)\not\subset\Aut$ be a continuous one-parameter semigroup in~$\UD$ with Denjoy\,--\,Wolff point ${\tau\in\partial\D}$ such that $\psi_1\in\Zen(\varphi) \setminus\{\id_\UD\}$.
    \begin{statlist}
		\item\label{IT_petalEST-hyperP}
Assume that $\Delta$ is a hyperbolic petal of $(\psi_t)$ with associated boundary repelling fixed point $\sigma\in\partial\D$. Then, one and only one of the following three statements holds:
		\begin{romlist}
          \item\label{IT_petalEST-hyperP(i)}
              $\varphi(\Delta)=\Delta$. In this case, there exists $t_0>0$ such that $\varphi=\psi_{t_0}$  and, in particular, $\sigma$ is a repelling fixed point of $\varphi$.
          \item\label{IT_petalEST-hyperP(ii)}
             $\Delta\cap\varphi(\Delta)=\emptyset$ and there exists a hyperbolic petal $\Delta^\prime$ of $(\psi_t)$ with associated  repelling fixed point $\sigma^\prime$ such that  $\varphi(\Delta)\subset\Delta^\prime$ and $\angle\lim_{z\to \sigma}\varphi(z)=\sigma^\prime$.
          \item\label{IT_petalEST-hyperP(iii)}
             $\Delta\cap\varphi(\Delta)=\emptyset$ and there exists a parabolic petal $\Delta^\prime$ of $(\psi_t)$ such that $\varphi(\Delta)\subset\Delta^\prime$ and $\angle\lim_{z\to \sigma}\varphi(z)=\tau$. In this case, $\sigma$ is an irregular contact point for~$\varphi$.
		\end{romlist}
\smallskip
		\item\label{IT_petalEST-paraP}
Assume that $\Delta$ is a parabolic petal of $(\psi_t)$. Then the following two statements hold:
		\begin{equilist}
			\item\label{IT_petalEST-paraP(a)}
                  $\varphi(\Delta)\subset \Delta$;
             \item\label{IT_petalEST-paraP(b)}
                  $\varphi(\Delta)=\Delta$ if and only if  $\varphi=\psi_{t_0}$ for some $t_0>0$.
		\end{equilist}
   \end{statlist}
\end{theorem}

In the hyperbolic case, by the reason mentioned in Remark~\ref{RM_hyperblic-simple},  only alternative~\ref{IT_petalEST-hyperP(i)} may occur, and moreover, \ref{IT_petalEST-paraP} becomes trivial because a hyperbolic semigroup cannot have parabolic petals.

In the parabolic case, based on Theorem~\ref{petalEST}, we are then able to give a quite complete answer to  Problem~\ref{theproblem2} as follows.
We again exclude from consideration the case $(\psi_t)\subset\Aut$, in which the result, with some obvious modifications,  holds trivially.
\begin{theorem}\label{Cor:petals} Let $\varphi\in\U(\D)$ be a non-elliptic self-map with Denjoy\,--\,Wolff point ${\tau\in\partial\UD}$ and let ${(\psi_t)\not\subset\Aut}$ be a continuous one-parameter semigroup in~$\UD$  such that ${\psi_1\in\Zen(\varphi)\setminus\{\id_\UD\}}$. Suppose that $\varphi$ is not an element of~$(\psi_t)$.  Then the following statements hold.
	\begin{Romlist}
        \item\label{IT_CorNew:if-hyperbolic-petal}
           Both $\varphi$ and $(\psi_t)$ are parabolic. Moreover, $(\psi_t)$ has at most one parabolic petal.
\medskip
		\item\label{IT_CorNew:petals-finite-seq}
            Suppose $(\psi_t)$ has a parabolic petal $\Delta_*.$ Then for each hyperbolic petal~$\Delta$ of~$(\psi_t)$, there exist a finite collection $\Delta_1,\,\Delta_2,...,\,\Delta_n\,$ of pairwise disjoint hyperbolic petals of $(\psi_t)$  such that
		$$
          \Delta_1=\Delta,\quad \varphi(\Delta_k)\subset \Delta_{k+1},\ k=1,...,n-1,\quad \varphi(\Delta_n)\subset\Delta_*.
		$$
         Moreover,
		$$
		  \varphi(\sigma_k)=\sigma_{k+1},\ k=1,...,n-1,\quad \varphi(\sigma_n)=\tau,
		$$
         where  $\sigma_k$ $(k=1,...,n)$ stands for the repelling fixed point associated with~$\Delta_k$.
\medskip
		\item\label{IT_CorNew:petals-inf-seq}
          Suppose that $(\psi_t)$ has no parabolic petal. Then for each hyperbolic petal~$\Delta$ of~$(\psi_t)$, there exists a sequence $(\Delta_n)$ of pairwise disjoint hyperbolic petals of $(\psi_t)$ such that
		$$
         \Delta_1=\Delta,\quad \varphi(\Delta_n)\subset \Delta_{n+1},\quad n\in\N. \qquad\qquad\qquad\qquad\quad\mbox{~}
		$$
            Moreover,
		$$
		  \varphi(\sigma_n)=\sigma_{n+1},\ n\in\N,\quad\text{~and}\quad \lim_{n\to\infty}\sigma_n=\tau,
		$$
          where  $\sigma_n$ $(n\in\Natural)$ stands for the repelling fixed point associated with~$\Delta_n$.
	\end{Romlist}
\end{theorem}

\begin{remark}
The above two theorems imply that every repelling fixed point~$\sigma$ of ${\psi:=\psi_1}$ is a contact point for~$\varphi$ and that $\varphi(\sigma)$ is also a boundary regular fixed point of~$\psi$. Moreover, if ${\varphi(\sigma)\neq\sigma}$, then the (forward) orbit of~$\sigma$ w.r.t. $\varphi$ (extended to a.e.\,point of~$\partial\UD$ by angular limits)  either hits the Denjoy\,--\,Wolff point within finite time, or it consists of pairwise distinct repelling fixed points of~$\psi$, but tends to the Denjoy\,--\,Wolff point in the limit. In part, this resembles the situation with two general commuting holomorphic self-maps $\varphi,\psi\in\Hol(\UD)$ studied by Bracci~\cite{Filippo-puntosTAMS}. In this general case, $\varphi$ can have a repelling cycle consisting of repelling fixed points of~$\psi$, but this possibility is ruled out if $\varphi$ is univalent and non-elliptic, see \cite[Proposition~5.2]{Filippo-puntosTAMS}. At the same time, the results established in~\cite{Filippo-puntosTAMS} do not seem to exclude some other situations not occurring in our case, such as a possibility that the orbit of~$\sigma$ hits within finite time a common fixed point of $\varphi$ and~$\psi$ different from their Denjoy\,--\,Wolff point. Moreover, in contrast to the context of Theorem~\ref{Cor:petals},  it is not clear in general whether any two orbits starting from different repelling fixed points of~$\psi$ necessarily fall into the same category.
\end{remark}

Theorems~\ref{petalEST} and~\ref{Cor:petals} are proved in Section~\ref{S_petalsII}. Further, in Section~\ref{PetalosISO}, we consider the last two alternatives given in Theorem~\ref{petalEST}\,\ref{IT_petalEST-hyperP} and analyse in which case equality ${\varphi(\Delta)=\Delta'}$ occurs. This will be used, along with several other results, in the proof of Theorem~\ref{Th_main} given in Section~\ref{S_proof-of-the-main-Theorem}. Finally, in Section~\ref{Sec:commutingsemigroups}, based on our findings, we substantially improve a result by Elin et al~\cite[Sect.\,5]{conReich} on sufficient conditions for two parabolic continuous one-parameter semigroups $(\varphi_t)$ and $(\psi_t)$ to commute, given that ${\varphi_1\circ\psi_1}={\psi_1\circ\varphi_1}$.

The part of the paper preceding Sections~\ref{S_petalsII}\,--\,\ref{Sec:elliptic case}, the content of which we have already described, is organized as follows. In Section~\ref{Notation}, we recall some preliminaries from holomorphic dynamics that will be needed to follow the paper. In Section~\ref{theexample}, we present in detail an example related to Problem~\ref{theproblem} and already mentioned above. In the same section, we establish a characterisation for holomorphic self-maps~$\varphi$ commuting with a non-elliptic semigroup~$(\psi_t)$ in terms of the Koenigs function of~$(\psi_t)$. In the next section, Section~\ref{S_affine}, we prove Theorem~\ref{TH_affine}, which yields another characterisation in terms of the boundary behaviour of $\psi_1$ and $\varphi$ along forward orbits under the discrete dynamics of~$\psi_1$. Section~\ref{petals} contains  a brief survey on the theory of petals and some lemmata relating commutativity to petals, which we use in the proof of Theorem~\ref{TH_FixP} given in Section~\ref{Sec6}, as well as in the proof of Theorems~\ref{petalEST} and~\ref{Cor:petals}.

\section{Notation and preliminaries}\label{Notation}

Below we introduce some notation and basic theory used further in the paper. For a more detail and for the proofs of the results presented in this section, we refer the interested readers to the recent monographs \cite{Abate2,BCD-Book}.

\subsection{Notation}
As usual, we denote the unit disc by
${\UD:=\{z\in\C:|z|<1\}}$,
and we write $\UH:={\{w\in\C:\Im w>0\}}$ for the upper half-plane and $\Hr:={\{z\in\C:\Re w>0\}}$ for the right half-plane.

Furthermore, denote by $\Hol(D,E)$ the class of all holomorphic mappings of a domain $D\subset\C$ into a set $E\subset\C$,
and let $\U(D,E)$ stand for the class of all \textit{univalent} (i.e. injective holomorphic) mappings from $D$ to~$E$.  As usual, we endow $\Hol(D,E)$ and $\U(D,E)$ with the topology of locally uniform convergence. In case $E=D$, we will write  $\Hol(D)$ and $\U(D)$ instead of $\Hol(D,D)$ and $\U(D,D)$, respectively.

For a self-map $\varphi:D\to D$ of a domain $D\subset\C$ and ${n\in\Natural}$ we denote by $\varphi^{\circ n}$ the $n$-th iterate of~$\varphi$, and let $\varphi^{\circ0}:=\id_D$. Moreover, if $\varphi$ is an automorphism of~$D$, then for every $n\in\N$, we denote by $\varphi^{\circ-n}$ the $n$-th iterate of~$\varphi^{-1}$.

\subsection{Holomorphic self-maps of the unit disc}
The study of the dynamics of an arbitrary holomorphic self-map $\varphi$ of the unit disc $\mathbb{D}$ is a classical and well-established branch of Complex Analysis.  An important role is played by the fixed points, all of which~--- except for at most one~--- lie on the boundary and hence should be understood in the sense of angular limits: ${\sigma\in\UC}$ is a \dff{boundary fixed point} of $\varphi\in\Hol(\UD)$ if $\varphi(\sigma):={\anglim_{z\to\sigma}\varphi(z)}$ exists and coincides with~$\sigma$. It is known that the angular derivative $\varphi'(\sigma)$ exists at every boundary fixed point~$\sigma$, but it can be infinite. In this latter case, $\sigma$ is referred to as a \dff{super-repelling} fixed point; otherwise, i.e. when $\varphi'(\sigma)$ is finite, it is in fact a positive real number and the boundary fixed point~$\sigma$ is said to be \dff{regular} (\dff{BRFP} for short).

The central result in the area is the Denjoy\,--\,Wolff Theorem, which states that if $\varphi$ is different from an elliptic automorphism (i.e. not an automorphism of~$\UD$ possessing a fixed point in~$\UD$), then the sequence of the iterates $(\varphi ^{\circ n})$ converges locally uniformly in~$\UD$ to a certain point~${\tau\in\overline{\mathbb{D}}}$.  This point  is called the \textit{Denjoy\,--\,Wolff point\/} of $\varphi$. Moreover, if $\tau\in \partial \D$, it is the unique boundary fixed point at which the angular derivative $\varphi'(\tau)$ is finite and belongs to $(0,1]$. In particular, for every BRFP~$\sigma$ different from the Denjoy\,--\,Wolff point~$\tau$, we have ${\varphi'(\sigma)}\in{(1,+\infty)}$. By this reason such points are referred to as \dff{repelling fixed points} of~$\varphi$.

According to the position of the Denjoy\,--\,Wolff point~$\tau$ and to the value of the \textsl{multiplier} $\varphi'(\tau)$, holomorphic self-maps $\varphi\in\Hol(\UD)$ different from elliptic automorphisms are divided into three categories. Namely, $\varphi$ is called:
\begin{itemize}
\item[(a)] \textit{elliptic\/} if $\tau\in\UD$,

\item[(b)] \textit{hyperbolic\/} if $\tau\in \partial \D$ and $\varphi'(\tau )<1$, and

\item[(c)] \textit{parabolic\/} if $\tau
\in \partial \D$ such that $\varphi'(\tau )=1$.
\end{itemize}
The identity mapping~$\id_\UD$ and all elliptic automorphisms of~$\UD$ are conventionally included in the category~(a) of elliptic self-maps.
Similarly, for an elliptic automorphism different from~$\id_\UD$, its Denjoy\,--\,Wolff point is defined to be its unique fixed point in~$\UD$.

Parabolic self-maps can have very different properties depending on the so-called \textit{hyperbolic step}.
Denote by $\rho_\D$ the hyperbolic distance in $\mathbb{D}$, and let $\varphi\in\Hol(\UD)$ be non-elliptic. Thanks to the Schwarz\,--\,Pick Lemma, for the orbit $\big(z_n\big):=\big(\varphi^{\circ n}(z_0)\big)$ of any point ${z_0\in\UD}$, there exists a finite limit $q(z_0):=\lim_{n\to+\infty} \rho_\D(z_{n},z_{n+1})$. It is known, see e.g. \cite[Corollary\,4.6.9]{Abate2}, that  either $q(z_0)>0$ for all~${z_0\in\UD}~$, or $q\equiv0$ in~$\UD$.  The self-map~$\varphi$ is said to be of \textit{positive} or of \textit{ zero hyperbolic step} depending on whether the former or the latter alternative occurs.
If $\varphi$ is
hyperbolic, then it is always of positive hyperbolic step. However, there exist parabolic self-maps of zero as well as of positive hyperbolic step.

\subsection{Holomorphic models for univalent self-maps}\label{Sec:modelos}

An indispensable role in our study is played by the concept of a holomorphic model, which goes back to Pommerenke~\cite{Pom79}, Baker and Pommerenke~\cite{BakerPommerenke}, and Cowen~\cite{Cowen} and which is discussed below for the special case of a univalent self-map. The terminology we use is mainly borrowed from~\cite{Canonicalmodel}.

\begin{definition}\label{DF_holomorphic-model} A \textit{holomorphic model} of $\varphi\in\U(\UD)$ is any triple $\mathcal M:=(S,h,\alpha)$, where $S$ is a Riemann surface, $\alpha$ is an automorphism of $S$, and $h$ is a univalent map from $\UD$ into $S$ satisfying the following two conditions:
	\begin{enumerate}[left=2.5em]
		\item[(HM1)] $h\circ \varphi=\alpha \circ h$, {}~and
		\item[(HM2)] $S\,=\,\bigcup_{n\geq0} \alpha^{\circ \, -n}(h(\D))$.
	\end{enumerate}
The Riemann surface $S$ is called the \textit{base space}, and the map $h$ is called the \textit{intertwining map} of the holomorphic model~$\mathcal M$.
\end{definition}

Every $\varphi\in\U(\UD)\setminus\{\id_\UD\}$ admits an essentially unique holomorphic model. More precisely, the following fundamental theorem holds.

\begin{theorem}[\protect{\cite[Theorem~1.1]{Canonicalmodel}}] \label{Thm:uniqness} Every $\varphi\in\U(\UD)$ admits a holomorphic model. Moreover such a model is unique up to a model isomorphism; i.e., if $(S_1,h_1,\alpha_1)$ and $(S_2,h_2,\alpha_2)$ are holomorphic models for $\varphi$, then there exists a biholomorphic map $\eta$ of~$S_1$ onto~$S_2$ such that
	$$
	h_2=\eta\circ h_1,\quad \alpha_2=\eta\circ\alpha_1\circ \eta^{-1}.
	$$
\end{theorem}

The type of a univalent self-map (elliptic, hyperbolic, or parabolic) is reflected in, and actually can be fully determined from the kind of holomorphic model $\varphi$ admits.  For an open interval $I\subset\Real$, we define
$$
 S_I:=\Real\times I=\{x+iy:x\in\Real,\,y\in I\}.
$$

\begin{theorem}[\cite{Cowen}, see also \cite{Canonicalmodel}]\label{Thm:model} Let $\varphi\in\U(\UD)\setminus\{\id_\UD\}$. The following statements hold.
\begin{ourlist}
	\item\label{IT_HM-ell-auto} $\varphi$ is an elliptic automorphism with multiplier $\lambda\in\partial\UD\setminus\{1\}$ if and only if $\varphi$ admits a holomorphic model of the form ${\mathcal M_\varphi:=(\UD,h,z\mapsto \lambda z)}$, where ${h\in\Aut}$.
	\item\label{IT_HM-ell-non-auto} $\varphi$ is an elliptic self-map with multiplier $\lambda\in\UD^*$ (and hence it is not an automorphism) if and only if $\varphi$ admits a holomorphic model of the form ${\mathcal M_\varphi:=(\C,h,z\mapsto \lambda z)}$.
	\item\label{IT_HM-hyp} $\varphi$ is a hyperbolic self-map with multiplier $\lambda\in(0,1)$ if and only if $\varphi$ admits a holomorphic model of the form $\mathcal M_\varphi:=(S_{I},h,z\mapsto z+1)$, where $I=(a,b)$ is a bounded open interval of length ${b-a=\pi/|\log\lambda|}$.
	\item\label{IT_HM-para-PHS} $\varphi$ is a parabolic self-map of positive hyperbolic step if and only if $\varphi$ admits a holomorphic model of the form ${\mathcal M_\varphi:=(S_{I},h,z\mapsto z+1)}$, where $I$ is an open unbounded interval different from the whole~$\Real$.
	\item\label{IT_HM-para-0HS} $\varphi$ is a parabolic self-map of zero hyperbolic step if and only if $\varphi$ admits a holomorphic model of the form ${\mathcal M_\varphi:=(\C,h,z\mapsto z+1)}$.	
\end{ourlist}	
\end{theorem}

\begin{remark}\label{RM_normalization}
In the above theorem, we may assume that:
\begin{itemize}
 \item[-]in case~\ref{IT_HM-ell-auto}, $h'(\tau)>0$, where $\tau$ is the Denjoy\,--\,Wolff point of $\varphi$;
 \item[-]in case~\ref{IT_HM-ell-non-auto}, $h'(\tau)=1$, where $\tau$ is the Denjoy\,--\,Wolff point of $\varphi$;
 \item[-]in cases~\ref{IT_HM-hyp} and~\ref{IT_HM-para-0HS}, $h(0)=0$;
 \item[-]in case~\ref{IT_HM-para-PHS}, $\Re h(0)=0$ and ${S_I=S_{(0,+\infty)}=\UH}$ or ${S_I=S_{(-\infty,0)}=-\UH}$.
\end{itemize}
Using the uniqueness part of Theorem~\ref{Thm:uniqness}, one can show (see e.g. \cite[Corollary 4.6.12]{Abate2} for details) that the above assumptions play the role of a normalization under which the holomorphic model $\mathcal M_\varphi$ for a given $\varphi\in\U(\UD)\setminus\{\id_\UD\}$ is unique. Note that the normalization for cases \ref{IT_HM-hyp} and~\ref{IT_HM-para-0HS} would also work in case~\ref{IT_HM-para-PHS}, but we prefer to use another normalization, so that for parabolic self-maps of positive hyperbolic step,  the base space~$S_I$ of~$\mathcal M_\varphi$ coincides with $\UH$ or~$-\UH$. Moreover, replacing, if necessary, $\varphi$ with $z\mapsto\overline{\varphi(\bar z)}$ we may assume that ${S_I=\UH}$.
\end{remark}

\begin{definition}\label{DF_canonical}
 The unique holomorphic model $\mathcal M_\varphi$ of a self-map  $\varphi\in\U(\UD)\setminus\{\id_\UD\}$ defined in Theorem~\ref{Thm:model} and normalized as in Remark~\ref{RM_normalization} is called \dff{canonical (holomorphic) model} for~$\varphi$. The intertwining map~$h$ of the canonical model~$\mathcal M_\varphi$ is called the \dff{Koenigs function}, and ${\Omega:=h(\UD)}$ is called the \textit{Koenigs domain} of~$\varphi$.
\end{definition}

\subsection{Commuting holomorphic self-maps}\label{SS_cummuting}
It is clear that if two  holomorphic self-maps $\varphi,\psi\in\Hol(\UD)\setminus\{\id_\UD\}$ commute, i.e. ${\varphi\circ\psi}={\psi\circ\varphi}$, and if one of them is elliptic, then the other is also elliptic and they share the Denjoy\,--\,Wolff point. The situation is not so evident when we consider non-elliptic self-maps. In 1973, Behan \cite{Behan}, see also \cite[Section 4.10]{Abate2}, proved that if $\varphi, \psi$ are non-elliptic self-maps of~$\UD$ with Denjoy\,--\,Wolff points $\tau_{\varphi}$ and $\tau_{\psi}$, respectively, then:
\begin{enumerate}
\item[(i)] if $\varphi$ is not a hyperbolic automorphism, then $\varphi$ and~$\psi$ share the  Denjoy\,--\,Wolff point, i.e.~$\tau_{\varphi}=\tau_{\psi}$;
\item[(ii)]  if $\varphi$ is a hyperbolic automorphism, then $\psi$ is a hyperbolic automorphism as well and it has the same fixed points as~$\varphi$.
\end{enumerate}

Later, Cowen proved that if $\varphi$ and $\psi$ are two non-elliptic commuting holomorphic self-maps of~$\D$ and if $\varphi$  is hyperbolic, then $\psi$ is also hyperbolic (and thus if $\varphi$ is parabolic, then $\psi$ is parabolic) \cite[Corollary~4.1]{Cowen-comm}, see also \cite[Theorem~1.3]{Simultaneous}.

In contrast to the Denjoy\,--\,Wolff point, a repelling fixed point~$\sigma$ of a holomorphic self-map~$\psi$ commuting with~$\varphi$ does not have to be a boundary fixed point of~$\varphi$: it is only guaranteed that $\sigma$ is a \dff{contact point} of~$\varphi$, i.e. the angular limit $\varphi(\sigma):={\anglim_{z\to\sigma}\varphi(z)}$ exists and belongs to~$\UC$, and that $\varphi(\sigma)$ is also a BRFP of~$\psi$; see e.g. \cite{Filippo-puntosTAMS} or \cite[Sect.\,4.10]{Abate2} and references therein. Clearly, the study of contact points reduces to that of boundary fixed points by composing with suitable rotations, and as a result one can classify contact points of a holomorphic self-map into  \dff{regular} and \dff{irregular} depending on whether the angular derivative at a given contact point is finite or not.

Restricted to \textbf{univalent} self-mappings of~$\UD$, in \cite{CDG-Centralizer} we studied the so-called \textsl{centralizers} of non-elliptic self-maps $\varphi\in\U(\UD)$ defined by
$$
  \Zen(\varphi):=\big\{\psi\in\U(\UD):\varphi\circ\psi=\psi\circ\varphi\big\}.
$$
Let ${(S,h,z\mapsto z+1)}$ be the canonical holomorphic model for the self-map~$\varphi$. Then, by \cite[Theorem~3.3]{CDG-Centralizer}, a self-map $\psi\in\U(\UD)$ commutes with~$\varphi$ if and only if $g_\psi:={h \circ \psi \circ h^{-1}\,:\,\Omega\to\Omega}:=h(\UD)$ extends holomorphically to a univalent self-map of~$S$ commuting with ${w\mapsto w+1}$.

Moreover,  if $\varphi$ is hyperbolic or parabolic of zero hyperbolic step, then the function $g_\psi$ corresponding to any ${\psi\in\Zen(\varphi)}$ is necessarily an affine map of the form ${w\mapsto w+c}$, where  $c$ is a (complex) constant; see \cite[Theorems~5.2 and~6.4]{CDG-Centralizer}. However, for the case of a parabolic self-map~$\varphi$ of positive hyperbolic step, the situation is quite different: there can be many $\psi$'s in the centralizer $\Zen(\varphi)$, for which $g_\psi$ is not an affine map.

\begin{definition}\label{Def:affine}
In the above setting, we say that $\psi\in\Zen(\varphi)$ is \textsl{affine w.r.t.~$\varphi$} if $g_\psi$ is an affine map.
\end{definition}
\begin{remark}\label{RM_affine-leadCoeff=1}
It is easy to see that if $\psi\in\Zen(\varphi)$ is affine w.r.t.~$\varphi$, then $g_\psi(w)=w+c$. Moreover, see \cite[Remark~4.3]{CDG-Centralizer}, the constant $c=c_{\varphi,\psi}$ is given by
\begin{equation}\label{EQ_formula-for-c}
c_{\varphi,\psi}=\begin{cases}
 \displaystyle \angle\lim_{z\to\tau}\big(\psi(z)-z\big)/\big(\varphi(z)-z\big),& \text{if $\varphi$ is parabolic,}\\[1.5ex]
 \hphantom{\angle}\frac{\log{\psi^\prime(\tau)}}{\log{\varphi^\prime(\tau)}}, & \text{if $\varphi$ is hyperbolic,}
 \end{cases}
\end{equation}
where $\tau$ stands for the Denjoy\,--\,Wolff point  of~$\varphi$.
\end{remark}

\subsection{One-parameter semigroups in the unit disc}\label{SS_one-param-semigr}
The study of one-parameter semigroups of holomorphic self-maps goes back to the early 1900s. However, one of the first pivot works, which brought this topic to the forefront, was published much later, namely in~1978, by Berkson and Porta~\cite{BP}, who studied continuous one-parameter semigroups of holomorphic self-maps of the unit disc~$\UD$ in connection with composition operators.  The current state of the art in this rich topic is presented in the monograph~\cite{BCD-Book}; see also~\cite{EliShobook10}.
We recall the definition straight away.
\begin{definition} \label{def:semigroup}
We say that a family $(\psi_t)_{t\geq0}$ (or to simplify,  $(\psi_t)$)
of holomorphic functions ${\psi_t:\D\to \D}$ is a \dff{one-parameter semigroup} if it verifies the following two algebraic properties:
\begin{itemize}
    \item[(i)] $\psi_0=\id_{\D}$;
    \item[(ii)] $\psi_t\circ\psi_s=\psi_{t+s}~$ for every $t,s\geq0$.
\end{itemize}
If, in addition, $\psi_t\to\psi_0$ uniformly on compact subsets of $\D$, as $t\to0^+$, we say that the one-parameter semigroup $(\psi_t)$ is \dff{continuous}.
\end{definition}

\smallskip\noindent{\bf Convention.}
From now on, unless explicitly indicated otherwise, by a continuous one-parameter semigroup we will mean a \textit{non-trivial} one, i.e. containing at least one element different from the identity map.
\smallskip

It is worth recalling that all elements of any continuous one-parameter  semigroup are univalent functions (see, e.g., \cite[Theorem~8.1.17]{BCD-Book}). Moreover, all of them, except for the identity map, have the same Denjoy\,--\,Wolff point, so one can talk about elliptic and non-elliptic continuous one-parameter semigroups (see, e.g., \cite[Theorem~8.3.1]{BCD-Book}). A similar remark concerns the repelling and super-repelling fixed points (see~\cite{CDP2004} and~\cite{Analytic-flows}).

Given a continuous one-parameter semigroup $(\psi_{t})$, it is possible to show that all the functions of the semigroup \textit{essentially} share their canonical model.
Indeed,  for non-elliptic semigroups, if $(S,h,z\mapsto z+1)$ is the canonical holomorphic model for $\psi_{1}$ given in Theorem~\ref{Thm:model}, then the triple $\big(S,h,(z\mapsto z+t)_{t\ge0}\big)$ is the canonical holomorphic model for~$(\psi_t)$ (see~\cite[Theorem~9.3.5]{BCD-Book} for the precise definition and further details). In particular, we have that
\begin{equation}\label{EQ_Abel-eq-for-semigroup}
h\circ \psi_{t}=h+t\qquad\text{for all $~t\geq 0$.}
\end{equation}
It follows  that $\psi_t={h^{-1}\circ(h+t)}$ for all ${t\ge0}$, we can differentiate this equality w.r.t.~$t$ and get ${\di\hskip-.075em\psi_t/\di t=G\circ\psi_t}$, where ${G:=1/h'}$ is called the \dff{infinitesimal generator} of~$(\psi_t)$.
(In a similar way, one can introduce the infinitesimal generator for an elliptic continuous one-parameter semigroup; see, e.g., \cite[Theorem~10.1.4\,(2)]{BCD-Book}. A detailed description of the properties of infinitesimal generators can be found in \cite[Chapter~10]{BCD-Book}.)

Moreover, in combination with Theorem~\ref{Thm:model}, the fact concerning holomorphic models of $\psi_t$'s mentioned above implies that the non-elliptic semigroup $(\psi_t)$, with the exclusion of ${\psi_0=\id_\UD}$, is entirely contained in one of the categories: hyperbolic holomorphic self-maps, parabolic holomorphic self-maps of zero hyperbolic step, or those of positive hyperbolic step\,--- and hence the whole semigroup can be classified accordingly.

\begin{definition}\textcolor{black}{
The function $h$ and the domain $\Omega:=h(\UD)$ defined above are called the \dff{Koenigs function} and \dff{Koenigs domain} of the non-elliptic semigroup~$(\psi_t)$.}
\end{definition}

\begin{remark}
 It is worth mentioning that if a continuous one-parameter semigroup~$(\psi_t)$ contains an automorphism of~$\UD$ different from~$\id_\UD$, then $(\psi_t)\subset\Aut$ and in such case we say that $(\psi_t)$ is a  \dff{group}, as it can be extended to a one-parameter group by setting ${\psi_t:=(\psi_{-t})^{-1}}$ for all~${t<0}$; see e.g. \cite[\S8.2]{BCD-Book} for more details.
\end{remark}

\section{The starting example and a characterization of self-maps commuting with a semigroup}\label{theexample}

As it was mentioned in the introduction, there are functions $\varphi\in \U(\D)$ and continuous one-parameter semigroups $(\psi_{t})$ in the unit disc such that $\varphi\circ\psi_{1}=\psi_{1}\circ \varphi$ and, at the same time, ${\varphi\circ\psi_{t}}\neq{\psi_{t}\circ \varphi}$ for some~${t>0}$. An example of this can be found in \cite[Example 8.5]{CDG-Centralizer}. Below we present a similar but more direct example.

\begin{example}\label{Mainexample}
Consider the complex domain
$$\Omega:=\C\setminus\Big(\bigcup_{p\in \Z}\Gamma_{p}\Big),\quad \Gamma_{p}:=\{x+ip:\, x\leq 0\}.$$
Denote by $h$ a Riemann map from~$\D$ onto~$\Omega$ and consider the continuous one-parameter semigroup $(\varphi_{t})$ formed by the functions $\varphi_t:={h^{-1}\circ(z\mapsto z+t)\circ h}$.

Note that $\Psi(w):=w+i$ is an automorphism of $\Omega$. Hence there exists a continuous one-parameter semigroup $(\Psi_{t})\subset\mathsf{Aut}(\Omega)$ such that ${\Psi_{1}=\Psi}$.
Now consider the continuous one-parameter semigroup~$(\psi_t)$ in~$\UD$ defined by $\psi_t:={h^{-1}\circ \Psi_{t}\circ h}$ for all ${t\ge0}$.

Suppose that ${\psi_t\circ\varphi_s}={\varphi_s\circ\psi_t}$ for some ${s,t>0}$. Then for any ${k\in\mathbb N}$ and all ${z\in\Omega}$, we have
\begin{equation}
 \Psi_t(z)=\Psi_t(z+ks)-ks.
\end{equation}
By this relation, $\Psi_t$ can be extended to a univalent function in~${\Omega-ks}$. Since ${k\in\Natural}$ is arbitrary, this in fact means that $\Psi_t$ extends as a univalent function to ${\bigcup_{k\in\Natural}(\Omega-ks)=\C}$. Therefore, $\Psi_t(w)=aw+b$ for all ${w\in\Omega}$ and some constants ${a\in\C\setminus\{0\}}$ and ${b\in\C}$. Taking into account that $\Psi_t$ is an automorphism of~$\Omega$, it further follows that ${a=1}$ and ${b=in}$ for some ${n\in\mathbb{Z}}$. Hence $\Psi_t=\Psi^{\circ n}$. Since ${t>0}$ and since the semigroup $(\Psi_t)$ is non-elliptic, the latter is possible only if ${t=n\in\Natural}$.

The above argument shows that although $\psi_1$ commutes with the whole semigroup~$(\varphi_t)$, the functions $\psi_t$ and~$\varphi_s$ do not commute whenever ${s>0}$ and ${t>0}$ is not integer. In particular, this shows that the answer to the question in Problem~\ref{theproblem} is negative: in general, the fact $\psi_1$ that commutes with $\varphi(:=\varphi_1)$ does not allow to conclude~--- if no additional condition is imposed~--- that the whole semigroup $(\psi_t)$ commutes with~$\varphi$.
\end{example}

\begin{remark}\label{RM_after-the-example}
Note that in the above example, the self-map $\varphi:=\varphi_1$ has infinitely many BRFPs $\sigma_k:=\lim_{x\to-\infty}h^{-1}\big(x+i(k+\tfrac12)\big)$, ${k\in\mathbb Z}$ \,(see \cite{Analytic-flows}). Although $\varphi$ commutes with $\psi_1$, none of these points is fixed by~$\psi_1$. In fact, $\psi_1(\sigma_k)=\sigma_{k+1}$ for all ${k\in\mathbb Z}$.
\end{remark}

The example presented above gives a hint on how to characterize holomorphic self-maps of~$\UD$ that commute with all elements of a given (non-elliptic) continuous one-parameter semigroup.

\begin{proposition}\label{PR_simple}
 Let $(\psi_t)$ be a non-elliptic continuous one-parameter semigroup with Koenigs function~$H$. For a holomorphic self-map $\varphi\in\Hol(\UD)$ the following two condition are equivalent:
 \begin{equilist}
 \item\label{IT_simple-commute} ${\varphi\circ\psi_t}={\psi_t\circ\varphi}$ for all $t>0$;
 \item\label{IT_simple-affine} there exists a constant ${c\in\Complex}$ such that
 $
   H\circ\varphi=H+c.
 $
 \end{equilist}
In particular, any $\varphi\in\Hol(\UD)$ satisfying~\ref{IT_simple-commute} is univalent.
\end{proposition}
\begin{proof}
Clearly, \ref{IT_simple-affine} implies~\ref{IT_simple-commute}. To prove the converse implication,
consider the holomorphic map ${g:\Omega\to\Omega:=H(\UD)}$ defined by $g:={H\circ\varphi\circ H^{-1}}$.
Condition~\ref{IT_simple-commute} is equivalent to
$$
 g(w+t)=g(w)+t\quad\text{for all~$~w\in\Omega~$ and all~$~t\ge0$.}
$$
Fix some $w_0\in\Omega$. The above identity means that for all $w\in\{w_0+t:t\ge0\}$, we have ${g(w)=w+c}$, where ${c:=g(w_0)-w_0}$. By the identity principle for holomorphic functions, ${g(w)=w+c}$ for all ${w\in\Omega}$, which is equivalent to~\ref{IT_simple-affine}.

Finally, if  \ref{IT_simple-affine} holds, then $\varphi=H^{-1}\circ(H+c)$ and it immediately follows that $\varphi$ is univalent.
\end{proof}

\begin{remark}
Differentiating both sides of the functional equation $H\circ\varphi=H+c$, one can write a further condition equivalent to conditions~\ref{IT_simple-commute} and~\ref{IT_simple-affine} in the above proposition:
 \begin{equilist}\addtocounter{enumi}{2}
 \item $G\circ\varphi=\varphi'G$, where $G$ stands for the infinitesimal generator of~$(\psi_t)$.
 \end{equilist}
\end{remark}
Note that in our main results established in this paper, it is not supposed that the \textit{whole} semigroup~${(\psi_t)}$ is known. In most of the cases, we only assume that $\psi_1$ is given. Therefore, in our situation, the above characterizations cannot be applied directly.  Nevertheless, Proposition~\ref{PR_simple} turns out to be very helpful in some of our arguments.

\section{Proof of Theorem \ref{TH_affine}}\label{S_affine}
We start by recalling a key result on holomorphic models for parabolic self-maps of positive hyperbolic step that we need and use strongly in the proof (restricting to the special case of a univalent self-map).
\begin{theorem}[\cite{Pom79}; see also {\cite[Lemma 2.2]{Poggi}}]
	\label{modeloHr} Let $f\in\U(\Hr)$ be parabolic of positive hyperbolic step with Denjoy\,--\,Wolff point $\infty$, and let $w_0\in\Hr$. Then the sequence ${(h_n)\subset\U(\Hr)}$ defined as
	$$
	h_n(w):=\frac{f^{\circ n}(w)-i\Im f^{\circ n}(w_0)}{\Re f^{\circ n}(w_0)}, \quad n\in \N,\ w\in\Hr,
	$$
	converges uniformly on compact subsets of $\Hr$ to a certain $h\in\U(\Hr)$ such that $h(w_0)=1$. Moreover:
\begin{romlist}
  \item\label{IT_Pommerenke1}  there exists $b=b(w_0)\in\R^*:=\R\setminus\{0\}$ such that $h\circ f=h+ib$;
  \item\label{IT_Pommerenke2} $(\Hr,h, w\mapsto w+ib)$ is a holomorphic model for~$f$.
\end{romlist}
\end{theorem}

\begin{proof}[\proofof{Theorem \ref{TH_affine}}]
Let $\tau\in\partial\D$ be the Denjoy\,--\,Wolff point of $\varphi$ and consider the standard Cayley map $C(z):=\frac{\tau+z}{\tau-z},\ z\in\D$. Moreover, define
$$
	f:=C\circ \varphi \circ C^{-1}\in\U(\Hr),\quad \Psi:=C\circ \psi\circ C^{-1}\in\U(\Hr).
$$
By the hypothesis, $\psi\neq\id_\D$ and it commutes with the parabolic self-map~$\varphi$. Therefore, by Behan's Theorem~\cite{Behan} and by a result of Cowen~\cite[Corollary~4.1]{Cowen-comm}, $\psi$ is  parabolic and has the same  Denjoy\,--\,Wolff point $\tau$ as~$\varphi$. As a consequence, both $f$ and $\Psi$ are parabolic self-maps of $\Hr$ with Denjoy\,--\,Wolff point at~$\infty$.

\StepP{\ref{IT_affine:limit-existence}}
    Fix $z\in\D$. A direct computation shows that
	$$
	 \frac{\psi(z)-z}{\varphi(z)-z}=\frac{1+ f(C(z))}{1+\Psi(C(z))}\,\frac{\Psi(C(z))-C(z)}{f(C(z))-C(z)}.
	$$
	Therefore, denoting $w:=C(z)\in\Hr$ and, for all $n\in\N$,
	\begin{equation}\label{EQ_fla-for-R_phi-psi}
      R_{\varphi,\psi}(\varphi^{\circ n}(z))=
      \frac{\psi(\varphi^{\circ n}(z))-\varphi^{\circ n}(z)}{\varphi(\varphi^{\circ n}(z))-\varphi^{\circ n}(z)}=
      \frac{1+ f(f^{\circ n}(w))}{1+\Psi(f^{\circ n}(w))}\,%
          \frac{\Psi(f^{\circ n}(w))-f^{\circ n}(w)}{f(f^{\circ n}(w))-f^{\circ n}(w)}.
	\end{equation}
    Now, denote $x_n:=\Re(f^{\circ n}(1))$ and $y_n:=\Im(f^{\circ n}(1)),\ n\in\N$. Since ${f\in\U(\Hr)}$ is parabolic of positive hyperbolic step, by Theorem~\refeq{modeloHr}\,\ref{IT_Pommerenke1}, the sequence $h_n:={(f^{\circ n}-iy_n)/x_n}$ converges locally uniformly in~$\Hr$ to a univalent function $h$ such that ${h\circ f=h+ib},$ where ${b\in\R^*}$ is some constant.
  Since $f$ and $\Psi$ commute, we have that as~$n\to+\infty$,
	\begin{equation}\label{EQ_lim1}
     \begin{array}{r@{}l}
      \displaystyle
      \frac{\Psi(f^{\circ n}(w))-f^{\circ n}(w)}{f(f^{\circ n}(w))-f^{\circ n}(w)}&{}=
             \displaystyle
             \frac{f^{\circ n}(\Psi(w))-f^{\circ n}(w)}{f^{\circ n}(f(w))-f^{\circ n}(w)}\\&{}=
             \displaystyle\vphantom{\int\limits_0^1}
             \frac{h_n(\Psi(w))-h_n(w)}{h_n(f(w))-h_n(w)}~\longrightarrow~ \frac{1}{ib}\big(h(\Psi(w))-h(w)\big)
    \end{array}
    \end{equation}
    locally uniformly in~$\Hr$.
By \cite[Theorem~5.1]{BetCS}, taking into account that \cite[eqn.\,(5.1)]{BetCS} is equivalent to \cite[eqn.\,(5.2)]{BetCS}, we have
    \begin{equation}\label{EQ_lim2}
      \lim_{n\to\infty}\frac{1+f^{\circ n}(f(w))}{1+f^{\circ n}(\Psi(w))}=1\qquad\text{for every~$~w\in\Hr$.}
    \end{equation}
    Since the range of $q_n:=(1+f^{\circ n}\circ f)/(1+f^{\circ n}\circ\Psi)$ is contained in ${\C\setminus (-\infty,0]}$, the sequence $(q_n)$ is normal in~$\Hr$ and hence, the above convergence is locally uniform in~$\Hr$.

Now, combining~\eqref{EQ_fla-for-R_phi-psi}, \eqref{EQ_lim1}, and~\eqref{EQ_lim2}, we see that the sequence $\big(R_{\varphi,\psi}\circ \varphi^{\circ n}\big)$ converges locally uniformly in~$\UD$ to the function
	\begin{equation}\label{EQ_lim-f}
      f_{\varphi,\psi}(z):=\frac{1}{ib}\big(h(\Psi(C(z)))-h(C(z))\big)%
                                      =\frac{1}{ib}\big(h(C(\psi(z)))-h(C(z))\big),\quad z\in\UD.
	\end{equation}

\StepP{\ref{IT_affine:affine-iff-const}} According to Theorem~\ref{modeloHr}\,\ref{IT_Pommerenke2}, ${\big(S_b,\frac{1}{ib}(h\circ C),z\mapsto z+1\big)}$, where $S_b:=\H$ if ${b<0}$ and $S_b:=-\H$ if ${b>0}$,  is a holomorphic model for $\varphi$. Taking into account the uniqueness of the canonical model, see Remark~\ref{RM_normalization}, this means that, up to an additive constant, $\frac{1}{ib}h\circ C$ coincides with the Koenigs function of~$\varphi$. In combination with~\eqref{EQ_lim-f}, this means that $\psi$ is affine with respect to $\varphi$ if and only if ${\frac{1}{ib}\big(h\circ C\circ\psi - h\circ C\big)}=f_{\varphi,\psi}$ is constant. It remains to notice that if this condition holds, then the value of the constant is given by formula~\eqref{EQ_formula-for-c}, i.e.
	$$
     f_{\varphi,\psi}\equiv c_{\varphi\,\psi}= \anglim_{z\to\tau}\big(\psi(z)-z\big)/\big(\varphi(z)-z\big).
    $$

\StepP{\ref{IT_affine:extension} and~\ref{IT_affine:extension-bis}}
We may suppose that the canonical holomorphic model~$\mathcal M_\varphi$ of~$\varphi$ is of the form ${(\H,h_\varphi,z\mapsto z+1)}$. The proof in the case $\mathcal M_\varphi={(-\H,h_\varphi,z\mapsto z+1)}$ is similar. According to \cite[Proposition~7.2]{CDG-Centralizer}, there exists $F\in\Hol(\D,\H\cup\R)$ such that:
	\begin{enumerate}
		\item[(i)] the function $\UH\ni w\mapsto g(w):=w+F(e^{2\pi iw})$ belongs to $\U(\H)$
                   and $g\big( h_\varphi(\D)\big)\subset h_\varphi(\D).$
		\item[(ii)] $\psi=h_\varphi^{-1}\circ g\circ h_\varphi$.
	\end{enumerate}
	As we have seen in the proof of~\ref{IT_affine:affine-iff-const}, $h_\varphi$ coincides up to an additive constant with $\frac{1}{ib}(h\circ C)$. Therefore, using~\eqref{EQ_lim-f}, for every ${w\in h_\varphi(\D)}$ we get
	\begin{eqnarray*}
      f_{\varphi,\psi}\circ h_\varphi^{-1}(w)&=&\frac{1}{ib}\big(h\circ C\circ \psi\circ h_\varphi^{-1}(w)-h\circ C\circ h_\varphi^{-1}(w)\big)\\
		&=&h_\varphi\circ \psi\circ h_\varphi^{-1}(w)-h_\varphi \circ h_\varphi^{-1}(w)=g(w)-w.
	\end{eqnarray*}
Thus, $f_{\varphi,\psi}\circ h_\varphi^{-1}$ is a restriction of the map ${\UH\ni w\mapsto G(w):=F(e^{2\pi i w})\in \H\cup\R}$. This immediately implies~\ref{IT_affine:extension} and shows that \ref{IT_affine:extension-bis} follows from~(i) and~(ii).
\end{proof}

\section{Petals and commutativity I. Auxiliary results}\label{petals}

\begin{definition}
Let $(\psi_t)$ be a continuous one-parameter semigroup of the unit disc. A connected component of  $\mathcal W^\circ$, the interior of the backward invariant set $\mathcal W:={\bigcap_{t\geq0}\psi_t(\D)}={\bigcap_{n\in\Natural}\psi_1^{\circ n}(\D)}$, is called a {\sl petal}  of $(\psi_t)$.
\end{definition}
It is clear that if $(\psi_t)$ is a group, then the backward invariant set $\mathcal W$ is the unit disc and thus, it has a unique petal. If $(\psi_t)$ is not a group with Denjoy\,--\,Wolff point $\tau$,  it was proved in \cite[Propositions 13.4.2 and 13.4.12]{BCD-Book} that:
\begin{enumerate}
\item Every petal $\Delta$ is a simply connected domain,
\item $\psi_t(\Delta)=\Delta$ for all $t\geq 0$ and $(\psi_t|_\Delta)$ is a continuous group of automorphisms of $\Delta$,
\item  $\tau\in \partial \Delta$, and
\item  there exists $\sigma\in\partial \D\cap\partial \Delta$ (possibly $\sigma=\tau$) such that the backward orbits $\psi_{-t}(z_0):= (\psi_t|_\Delta)^{-1}(z_0)\to\sigma$ as $t\to+\infty$ for every ${z_0\in \Delta}$.  Moreover, $\sigma$ is a BRFP of the semigroup, and $\partial\Delta$ contains no boundary fixed points of~$(\psi_t)$ other than $\sigma$ and~$\tau$. We say that the petal $\Delta $ is \dff{parabolic} if $\sigma=\tau$. Otherwise, we say that the petal is \dff{hyperbolic}.
\item  If $(\psi_{t})$ has a parabolic petal, then $(\psi_t)$ is parabolic.
\item Every backward orbit $[0,+\infty)\ni t\mapsto z_t:=\psi_t^{-1}(z_0)$ starting from a point~${z_0\in\mathcal W^\circ}$ is \dff{regular}, meaning that ${\sup_{t\ge0}\rho_\UD(z_t,z_{t+1})<+\infty}$.
\item Suppose $\sigma\in \partial \D$ is a boundary regular fixed point of $(\psi_t)$ different from~$\tau$. Then there exists a unique hyperbolic petal $\Delta$ such that $\sigma\in \partial \Delta$.
\end{enumerate}
Following the conventional terminology in dynamical systems, we call the above BRFP~$\sigma$ the \hbox{\dff{$\alpha$-point}} of the corresponding hyperbolic petal~$\Delta$. In case of a parabolic petal, by its $\alpha$-point we mean the Denjoy\,--\,Wolff point~$\tau$.

In the rest of this section we suppose that $\varphi\in\U(\D)$ is a non-elliptic self-map with Denjoy\,--\,Wolff point $\tau\in\partial\D$, that $\varphi$ is different from a hyperbolic automorphism, and that $(\psi_t)$ is a continuous one-parameter semigroup of~$\UD$ satisfying $\id_\UD\neq\psi_1\in\Zen(\varphi)$.

\begin{lemma}\label{LM_petal-into-petal}
In the above assumptions, $\varphi(\mathcal W)\subset \mathcal W$, where $\mathcal W$ stands for the backward invariant set of~$(\psi_t)$. In particular, the image $\varphi(\Delta)$  of any petal~$\Delta$ of~$(\psi_t)$ is contained again in some petal of~$(\psi_t)$ (which may coincide with~$\Delta$).
\end{lemma}
\begin{proof}
Denote by $h$ the Koenigs function of~$(\psi_t)$.
Using Abel's equation~\eqref{EQ_Abel-eq-for-semigroup}, and taking into account that $h$ is univalent, it is easy to see that
\begin{equation}\label{EQ_h-of-invariant}
  h(\mathcal W)~=~\bigcap_{t\ge0}\,h(\UD)+t~=~\bigcap_{n\in\Natural}\,h(\UD)+n.
\end{equation}
The second equality holds because, again by~\eqref{EQ_Abel-eq-for-semigroup}, ${h(\UD)+t_2\subset h(\UD)+t_1}$ whenever ${t_2\ge t_1\ge0}$.

By the hypothesis $\varphi\in\Zen(\psi_1)$. Therefore, by \cite[Theorem~3.4]{CDG-Centralizer}, there exists a univalent holomorphic map ${g:h(\UD)\to h(\UD)}$ such that ${h\circ\varphi=g\circ\varphi}$ and ${g(w+1)}={g(w)+1}$ for any ${w\in h(\UD)}$. Therefore,
\begin{equation}\label{EQ_h-of-varphi-of-invariant}
  h\big(\varphi(\mathcal W)\big)~=~g\big(h(\mathcal W)\big)~=~g\,\Big(\,\bigcap_{n\in\Natural}\,h(\UD)+n\,\Big)
   ~=~\bigcap_{n\in\Natural}\,g\big(h(\UD)+n\big)~=~\bigcap_{n\in\Natural}\,g\big(h(\UD)\big)+n.
\end{equation}

Combining the fact that $g\big(h(\UD)\big)\subset h(\UD)$ with the equalities~\eqref{EQ_h-of-invariant} and~\eqref{EQ_h-of-varphi-of-invariant}, we see that ${h\big(\varphi(\mathcal W)\big)}\subset{h(\mathcal W)}$. Since $h$ is univalent, this proves the inclusion ${\varphi(\mathcal W)\subset\mathcal W}$. Finally, the statement regarding the petals follows immediately from this inclusion because the map $\varphi$ is continuous and open.
\end{proof}

Note that the petals of~$(\psi_t)$ are pairwise disjoint. Therefore, in Lemma~\ref{LM_petal-into-petal} proved above, the petal~$\Delta'$ that contains $\varphi(\Delta)$ is unique.
\begin{lemma}\label{LM_correspondence of BRFPs}
Let $\Delta'$ be the petal of~$(\psi_t)$ defined as above, i.e. $\varphi(\Delta)\subset\Delta'$. Then
$$
 \varphi(\sigma):=\anglim_{z\to\sigma}\varphi(z)\,=\,\sigma',
$$
where $\sigma$ and $\sigma'$ stand for the $\alpha$-points of the petals~$\Delta$ and $\Delta'$, respectively.
\end{lemma}
\begin{proof}
Fix any point $z_0\in\Delta$. Then according \cite[Proposition~13.4.2]{BCD-Book}, there is a unique regular backward orbit ${(z_t)\subset\Delta}$ of~$(\psi_t)$ starting from~$z_0$ and converging to~$\sigma$ as ${t\to+\infty}$. More precisely, for each ${t\ge0}$ there is a unique point ${z_t\in\Delta}$  such that ${\psi_t(z_t)=z_0}$ and the map ${[0,+\infty)\ni t\mapsto z_t}$ is continuous and injective, with $\lim_{t\to+\infty}z_t=\sigma$ and
 \begin{equation}\label{EQ_bdd-hyperbolic-step}
	\sup_{t\ge0} \rho_\D(z_t,z_{t+1})<+\infty.
 \end{equation}

Since $\varphi$ commutes with $\psi_1$ and hence with $\psi_n$ for any ${n\in\Natural}$, we have
$$
\psi_n\big(\varphi(z_n)\big)=\varphi\big(\psi_n(z_n)\big)=\varphi(z_0).
$$
Taking into account that $z_n\in\Delta$ and hence ${\varphi(z_n)\subset \Delta'}$ for all ${n\in\Natural\cup\{0\}}$, we see that the points $w_n$ lie on the backward orbit of~$w_0:=\varphi(z_0)\in\Delta'$. Therefore, ${w_n\to\sigma'}$ as ${n\to+\infty}$.

Thus, we see that
$\varphi(z_n)\to\sigma'$ for a sequence $(z_n)$ converging to~$\sigma$ and such that  the hyperbolic distance $\rho_\UD(z_n,z_{n+1})$, according to~\eqref{EQ_bdd-hyperbolic-step}, is bounded. This leads, see e.g. \cite[Theorem 1.4]{Canonicalmodel}, to the desired conclusion that $\varphi$ has angular limit at~$\sigma$ equal to~$\sigma'$.
\end{proof}

\begin{lemma}\label{LM_parabolic-petal-into-itself}
If $\Delta$ is a parabolic petal of $(\psi_t)$, then $\varphi(\Delta)\subset\Delta$.
\end{lemma}
\begin{proof}
By the definition of a parabolic petal, the $\alpha$-point~$\sigma$ of the petal~$\Delta$ coincides with  the Denjoy\,--\,Wolff point~$\tau$ of the semigroup~$(\psi_t)$, which in turn coincides with the Denjoy\,--\,Wolff point of~$\varphi$. Therefore, ${\varphi(\sigma)=\sigma}$. According to Lemma~\ref{LM_correspondence of BRFPs}, it follows that the petal~$\Delta'$ containing $\varphi(\Delta)$ must be also parabolic.

By \cite[Theorem~13.5.7]{BCD-Book}, $(\psi_t)$ has at most two parabolic petals. If it has exactly one parabolic petal, then we are done. If there are two distinct parabolic petals, then (again according \cite[Theorem~13.5.7]{BCD-Book}) the images $P_k$, $k=1,2$, of the two parabolic petals w.r.t. the Koenigs map~$h$ of~$(\psi_t)$ are of the form ${P_1=\{w:\Im w<a\}}$ and ${P_2=\{w:\Im w>b\}}$ for some ${b\ge a}$. As a consequence, the semigroup $(\psi_t)$, and hence the self-map $\psi_1$, are of zero hyperbolic step. Since ${\varphi\in\Zen(\psi_1)}$, by \cite[Proposition~4.3]{CDG-Centralizer}, $\varphi$ is affine w.r.t.~$\psi_1$, i.e. ${h\circ\varphi}={h+c}$ for a suitable ${c\in\C}$. It follows that if $\Delta$ and $\Delta'$ were two different parabolic petals, then the translation by $c$ or $-c$ would map $P_1$ into~$P_2$, which is clearly impossible. Thus, ${\varphi(\Delta)\subset\Delta'=\Delta}$.
\end{proof}

\begin{corollary}
\label{CR_petals-iff-points}
In the notation of Lemma~\ref{LM_correspondence of BRFPs}, $\varphi(\Delta)\subset\Delta$ if and only if ${\varphi(\sigma)=\sigma}$.
\end{corollary}
\begin{proof}
If $\Delta$ is a hyperbolic petal, then the corollary follows from Lemma~\ref{LM_petal-into-petal} and the fact that any two distinct petals have distinct $\alpha$-points unless both petals are parabolic.

So suppose $\Delta$ is a parabolic petal. Then $\varphi(\Delta)\subset\Delta$ by Lemma~\ref{LM_parabolic-petal-into-itself}. Moreover, as we have seen in the proof of Lemma~\ref{LM_parabolic-petal-into-itself}, ${\varphi(\sigma)=\sigma}$ because in this case, $\sigma=\tau$ is the Denjoy\,--\,Wolff point for~$(\psi_t)$ and hence also for~$\varphi$. These two observations complete the proof.
\end{proof}

\begin{lemma}\label{LM-only-one-parabolic-petal}
 Suppose that $\varphi$ is not an element of~$(\psi_t)$. Then $(\psi_t)$ can have at most one parabolic petal.
\end{lemma}
\begin{proof}
Suppose that $(\psi_t)$ has two distinct parabolic petals. Then, as in the proof of Lemma~\ref{LM_parabolic-petal-into-itself}, we see that ${h\circ\varphi}={h+c}$, where the constant~$c$ must be real because each of the half-planes $P_1$ and $P_2$ are mapped by the translation ${w\mapsto w+c}$ into itself. Moreover, ${h(\UD)+c}={h\big(\varphi(\UD)\big)}\subset h(\UD)$ and hence, by \cite[Theorem~3.1\,(B)]{CDG-Centralizer}, ${c\ge0}$. It follows that ${\varphi=\psi_c}$, which contradicts the hypothesis.
\end{proof}

We conclude this section with a lemma, which is definitely known for specialists and holds also for non-univalent self-maps, but lacking a reference we prefer to include a short proof making use of commutativity.
\begin{lemma}\label{LM_uniqueness}
Let $\varphi_1,\varphi_2\in\U(\UD)$ be such that $\varphi_1^{\circ k}=\varphi_2^{\circ k}$ for some ${k\in\Natural}$. If $\varphi_1$ is non-elliptic, then~${\varphi_2=\varphi_1}$.
\end{lemma}
\begin{proof}
Since $\varphi_1$ is univalent and non-elliptic, so is $\varphi_3:=\varphi_1^{\circ k}$. Let ${(S,h,z\mapsto{z+1})}$ stand for the canonical holomorphic model of~$\varphi_3$. Clearly, both $\varphi_1$ and $\varphi_2$ commute with $\varphi_3$. Therefore, see e.g. \cite[Theorem~3.4]{CDG-Centralizer}, $\varphi_m={h^{-1}\circ g_m\circ h}$, ${m=1,2}$, with certain ${g_m\in\U(S)}$ satisfying $g_m(w+1)={g_m(w)+1}$ for all ${w\in S}$. The equality $\varphi_2^{\circ k}=\varphi_1^{\circ k}=\varphi_3$ implies that
$$
  g_m^{\circ k}(w)=w+1,\quad m=1,2,\quad\text{for all}~w\in S.
$$
In particular, $g_1,g_2\in\mathsf{Aut}(S)$. Using the one-to-one correspondence between the automorphisms of~$S$ and those of~$\UD$ induced by a conformal mapping between $S$ and~$\UD$ and taking into account that the iterates of any $f\in\Aut$ have the same fixed points as~$f$, it is not difficult to see that both $g_1$ and~$g_2$ are translations along~$\Real$. Since ${g_1^{\circ k}=g_2^{\circ k}}$, it follows that ${g_1=g_2}$. Thus ${\varphi_2=\varphi_1}$.
\end{proof}

\section{Proof of Theorem \ref{TH_FixP}}\label{Sec6}
Recall that by the hypothesis $\varphi\in\U(\UD)$ is non-elliptic and different from a hyperbolic automorphism and that ${\psi_1\in\Zen(\varphi)\setminus\{\id_\UD\}}$. In particular, it follows that the semigroup~$(\psi_t)$ is  non-trivial and non-elliptic. Moreover, by Behan's Theorem, $\varphi$ and $(\psi_t)$ have the same Denjoy\,--\,Wolff point~$\tau$.

Denote by $(S,H,(z\mapsto z+t)_{t\geq 0})$ the canonical holomorphic model of~$(\psi_t)$.
\begin{proof}[\proofof{$\ref{IT_FixP-in-the-semigroup} \Rightarrow \ref{IT_FixP_wholeCOPS-contained}$}\nopunct\!] is completely trivial.
\phantom\qedhere
\end{proof}	

\begin{proof}[\proofof{$\ref{IT_FixP_wholeCOPS-contained}\Rightarrow \ref{IT_FixP_common}$}]
Since $(\psi_t)\subset\Zen(\varphi)$, by Proposition~\ref{PR_simple} we have ${H\circ \varphi}={H+c}$ for some ${c\in\C}$. The Koenigs function~$H$ has angular limits at all points on~$\UC$, possibly except for the Denjoy\,--\,Wolff point, see e.g. \cite[Corollary~11.1.7]{BCD-Book}. Let us check that $H(\sigma):={\anglim_{z\to\sigma}H(z)}=\infty$. Suppose that ${H(\sigma)\neq\infty}$. Then ${H\big(\varphi(r\sigma)\big)}={H(r\sigma)+c}{\to H(\sigma)+c}$ as ${r\to 1^-}$. By the hypothesis, $\sigma$ is a boundary fixed point of~$\varphi$, i.e. ${\varphi(r\sigma)\to\sigma}$ as ${r\to 1^-}$. According to the Lehto\,--\,Virtanen's Theorem (see e.g. \cite[Theorem~3.3.1]{BCD-Book}) it follows that
$$
 H(\sigma)=\lim_{r\to1^-}H\big(\varphi(r\sigma)\big)=H(\sigma)+c
$$
and hence  ${c=0}$, i.e. ${\varphi=\id_\UD}$. The latter contradicts the hypothesis. Thus, ${H(\sigma)=\infty}$ and, as a consequence, see e.g. \cite[Proposition~13.6.1]{BCD-Book}, $\sigma$ is a boundary fixed point of~$(\psi_t)$.
\phantom\qedhere
\end{proof}

\begin{proof}[\proofof{$\ref{IT_FixP_common} \Rightarrow \ref{IT_FixP-in-the-semigroup}$}]
By an old result of Heins \cite[Lemma~2.1]{Heins}, $\psi_1$ cannot be a hyperbolic automorphism because it commutes with $\varphi$, which is not a hyperbolic automorphism. Moreover, $\psi_1$ is not a parabolic automorphism because by \ref{IT_FixP_common}, it has a boundary fixed point $\sigma$ different from its Denjoy\,--\,Wolff point. It follows that the non-elliptic semigroup~$(\psi_t)$ cannot be extended to a group.
According to the regularity of~$\sigma$ we distinguish two cases.

\StepC{I}{$\sigma$ is a repelling fixed point of $(\psi_t)$} Let $\Delta$ be the hyperbolic petal of $(\psi_t)$ associated with $\sigma$, as explained in Section~\ref{petals}. Since $(\psi_t)$ is not a group, according to \cite[Theorems 13.2.7, 13.4.12]{BCD-Book}, there exists a univalent map $g$ from $\D$ onto $\Delta$ such that the formula $\widehat{\psi}_t:={g^{-1}\circ\psi_t\circ g}$,  ${t\geq0}$, defines a hyperbolic group~$(\widehat{\psi}_t)$  in~$\UD$ with Denjoy\,--\,Wolff point at~$-\sigma$.
	By Corollary~\ref{CR_petals-iff-points}, ${\varphi(\Delta)\subset\Delta}$. Therefore, $\widehat{\varphi}:={g^{-1}\circ\varphi\circ g}$ is an element of~$\U(\UD)$.
	Moreover,  $\psi_1\in\Zen(\varphi)$ implies that ${\widehat{\varphi}\in\Zen(\widehat{\psi}_1)}$. Since $\widehat{\psi}_1$ is a hyperbolic automorphism and by \cite[Remark 6.4]{CDG-Centralizer}, we conclude that there exists $t_0>0$ such that either $\widehat{\varphi}=\widehat{\psi}_{t_0}$ or $\widehat{\varphi}=\widehat{\psi}^{-1}_{t_0}$.
	If the latter alternative  would occur, then  $(\psi_{t_0}\circ \varphi)|_\Delta=\id_\Delta$ and hence, by the identity principle, ${\psi_{t_0}\circ \varphi=\id_\D}$, which is impossible because $\psi_1\not\in\Aut$. Therefore, $\widehat{\varphi}=\widehat{\psi}_{t_0}$.
	It immediately follows that $\psi_{t_0}|_\Delta=\varphi|_\Delta$ and hence, again by the identity principle, $\psi_{t_0}=\varphi$ as desired.

\StepC{II}{$\sigma$ is a super-repelling fixed point of $(\psi_t)$} According to \cite[Corollary~13.6.7]{BCD-Book}, we know that there is ${a\in \R}$ such that
\begin{equation}\label{Eq:super-repelling}
\lim_{z\to\sigma} \Im H(z)=a\quad \textrm{and }\quad \lim_{z\to\sigma} \Re H(z)=-\infty.
\end{equation}
Note that both limits are  unrestricted. Taking into account that $\sigma$ is a boundary fixed point of~$\varphi$, we therefore have
\begin{equation}\label{Eq:super-repelling2}
a=\lim_{r\to 1^{-}} \Im H(\varphi(r\sigma)).
\end{equation}
If $\psi_{1}$ is hyperbolic or parabolic of zero hyperbolic step,
then by \cite[Proposition~4.3]{CDG-Centralizer}, ${\varphi\in\Zen(\psi_{1})}$ implies that ${H\circ\varphi}={H+c}$ for a suitable ${c\in\C}$ and, in view of~\eqref{Eq:super-repelling} and \eqref{Eq:super-repelling2}, it follows that 	\begin{equation*}
		\Im c=\lim_{r\to1^{-}}\big(\Im H(\varphi(r\sigma))-\Im H(r\sigma)\big)=a-a=0.
	\end{equation*}
	Therefore, $c$ is a real number. Certainly $c\neq0$ because $\varphi\neq\id_\UD$. Moreover, if $c$ were negative, we would obtain that $\psi_{-c}\circ\varphi=\id_\D$, which is impossible because ${\psi_t\not\in\Aut}$ for any ${t>0}$. Hence, $c>0$ and it follows that ${\varphi=\psi_{t_0}}$ with ${t_0:=c}$. This completes the proof in the case when $(\psi_t)$ is hyperbolic or parabolic of zero hyperbolic step.

From now on we suppose that $\psi_1$ is parabolic of positive hyperbolic step.
 For the sake of clarity, we will also assume that the base space for $(\psi_t)$ is $S=\H$. The proof in the other case, i.e. for ${S=-\H}$, is completely similar.
	Then the fact that $\varphi\in\Zen(\psi_1)$ implies, by \cite[Proposition~7.2]{CDG-Centralizer}, that there exists $F\in \Hol(\D,\H\cup\R$) such that
	$$
	H\circ\varphi(z)=H(z)+F\big(e^{2\pi i H(z)}\big)\quad\text{for all~}~z\in\D.
	$$
We notice that the function $w\mapsto w+F(e^{2\pi iw})$ is univalent in $\H$, a fact that will be used later.

	By \eqref{Eq:super-repelling} and \eqref{Eq:super-repelling2}, we have
	\begin{equation}\label{EQ_from-Im-part}
		\lim_{r\to1^{-}}\Im F(e^{2\pi iH(r\sigma)})=\lim_{r\to1^{-}}\big(\Im H(\varphi(r\sigma))-\Im H(r\sigma)\big)=a-a=0.
	\end{equation}
By~\eqref{Eq:super-repelling}, for any ${p=e^{2\pi i \theta}}\in\partial\D$ there exist a natural $N=N(p)$ and a sequence $(r_n)=(r_n(p))$ in $(0,1)$ convergent to $1$ such that ${\Re H(r_n\sigma)=\theta-n}$ for all ${n\geq N}$. For this sequence~$(r_n)$, we have
	\begin{equation}\label{EQ_from-Im-part2}
		\lim_{n\to\infty}e^{2\pi i H(r_n \sigma)}=\lim_{n\to\infty}e^{2\pi i(\theta-n)}e^{-2\pi \Im H(r_n \sigma)}=pe^{-2\pi a}.
	\end{equation}
Since $H(\UD)\subset\UH$,  we have ${a\ge0}$. Below we consider separately the cases ${a>0}$ and ${a=0}$.
	
\StepG{Subcase}{ II.a: $a>0$.} In this case, $e^{-2\pi a}\in (0,1)$. Let $C_a$ be the circle centred at zero and radius $e^{-2\pi a}$. Taking into account that $C_a\subset\D$ and using \eqref{EQ_from-Im-part} and \eqref{EQ_from-Im-part2}, we deduce that ${\Im F(z)=0}$ for all ${z\in C_a}$. By the maximum principle for harmonic functions and the identity principle for holomorphic functions, it follows that ${F\equiv r}$ for some constant ${r\in\Real}$. Hence,  ${H\circ\varphi}={H+r}$. By essentially the same argument as in the case of zero hyperbolic case, we see that ${r>0}$ and therefore, ${\varphi=\psi_{t_0}}$ with~${t_0:=r}$, as desired.
	
\StepG{Subcase}{ II.b: $a=0$.}  Since $F\in \Hol(\D,\H\cup\R$), we have $\Im F$ is a non-negative harmonic function in the unit disc. By \cite[Corollary on p.\,38]{Hoffman}, there exists $E\subset\D$ of null measure such that the angular limit $\angle \lim_{z\to p} \Im F(z)$ exists  for every $p\in\partial\D\setminus E$.
By \eqref{EQ_from-Im-part2}, for every $p\in \partial \D$,
	the sequence $(e^{2\pi iH(r_n(p)\sigma)})$ converges non-tangentially to $p$. Therefore, by \eqref{EQ_from-Im-part}, we deduce that
	\begin{equation}\label{EQ_from-Im-part3}
		\angle \lim_{z\to p}\Im F(z)=0,\ \mathrm{for\ all\ }  p\in\partial\D\setminus E.
	\end{equation}
	Now, let us consider the function in $\Hol(\H,\C)$ given by $G(w):=F(e^{2\pi iw}),\ w\in\H$.
Similarly to the case ${a>0}$, it suffices to show that $G$ is a real constant. Consider $g(w):={w+G(w)}$. This function is a univalent holomorphic self-map of~$\UH$ with $g'(\infty):={\anglim_{w\to\infty}g(w)/w=1}$. Therefore, $g$ admits the following representation, see e.g. \cite[Chapter~V.4:\,(V.42),\,(V.44)\,\,(V.45)]{Bhatia},
\begin{multline*}
g(w)=r+w+\int_\Real\left(\frac1{\xi-w}-\frac{\xi}{1+\xi^2}\right)\di\mu(\xi),~w\in\UH,\quad\text{with~a constant~$r\in\Real~$ and }\\ \text{a non-negative Borel measure~$\mu$ on~$\Real$ given by~}\,
       \mu([a,b])=\lim_{\eta\to0^+}\frac1\pi\int_a^b\Im g(\xi+i\eta)\di\xi
\end{multline*}
for any $a,b\in\Real$, $a<b$, possibly except for a countably many points at which $\mu$ has atoms. Taking into account that by~\eqref{EQ_from-Im-part3}, $\lim_{\eta\to0^+}\Im g(\xi+i\eta)=0$ for a.e.~${\xi\in\Real}$, it remains to see that $\Im g$ is bounded in~$\UH\setminus\UH_1$, where for ${b\in\Real}$ we denote by $\UH_b$ the half-plane ${\{w\colon \Im w>b\}}$.

Note that ${g'(\infty)\neq0}$ and hence $g$ is conformal at~$\infty$ with ${g(\infty)=\infty}$. Recall that $g$ is univalent in~$\UH$. Then by a standard argument, see e.g. \cite[p.\,303--304]{Pombook75}, we get that $g(\UH_1)\supset U_R:=\{w:\arg w\in(\pi/2,3\pi/2), |w|>R\}$ for some~${R>0}$. Let $R_0:=\max\{1,R\}$. Since ${g(w+1)=g(w)+1}$ for all ${w\in\UH}$, it actually follows that ${\UH_{R_0}}\subset{\bigcup_{k\in\mathbb Z}U_{R_0}+k}\subset g(\UH_1)$. Therefore, ${g\big(\UH\setminus\UH_1\big)}={g(\UH)\setminus g(\UH_1)}\subset{\UH\setminus\UH_{R_0}}$, and we are done.
\end{proof}

\section{Petals and commutativity II.  Proof~of Theorems \ref{petalEST} and \ref{Cor:petals}}\label{S_petalsII}
Before we start proving  Theorem~\ref{petalEST}, let us analyze (in the remark below) the case $(\psi_t)\subset\Aut$, which in the statement of the theorem is excluded from consideration. As it has been already mentioned in Section~\ref{SS_one-param-semigr}, in this case, we can (and do) extend $(\psi_t)$ to a group by setting $\psi_t:=\big(\psi_{-t}\big)^{-1}$ for all ${t<0}$.

\begin{remark}\label{RM_group}
Suppose that $(\psi_t)\subset\Aut$ is a continuous one-parameter group in~$\UD$ with $\psi_1\neq\id_\UD$, and let $\varphi\in\U(\UD)$ be a non-elliptic self-map commuting with~$\psi_1$. Clearly, $(\psi_t)$ cannon be elliptic. The whole unit disc $\UD$ is the unique petal of~$(\psi_t)$, which is hyperbolic or parabolic depending on whether $(\psi_t)$ is a hyperbolic or a parabolic group.
Moreover, if $(\psi_t)$ is hyperbolic, then by \cite[Lemma~2.1]{Heins}, $\varphi$ is a hyperbolic automorphism having the same fixed points as~$(\psi_t)$, and hence ${\varphi=\psi_{t_0}}$ for some ${t_0\in\Real}$. If $(\psi_t)$ is parabolic, then $\varphi$ is not necessarily an automorphism of~$\UD$, but it has to be a parabolic self-map with the same Denjoy\,--\,Wolff point as~$\psi_1$; see, e.g., \cite[Proposition~2.6.11]{Abate2}. In this case, $\varphi\in\Aut$ if and only if ${\varphi=\psi_{t_0}}$ for some~${t_0\in\Real}$.
\end{remark}

\begin{proof}[\proofof{Theorem \ref{petalEST}}] First of all, since $\psi_1$ is not a hyperbolic automorphism (otherwise, we would have ${(\psi_t)\subset\Aut}$), by \cite[Lemma~2.1]{Heins}, $\varphi$ cannot be a hyperbolic automorphism either. Therefore, by Behan's Theorem, see \cite{Behan} or \cite[Theorem~4.10.3]{Abate2}, the Denjoy\,--\,Wolff point of~$\varphi$ coincides with the Denjoy\,--\,Wolff point~$\tau$ of the semigroup~$(\psi_t)$. In particular, it follows that $(\psi_t)$ is non-elliptic.

\StepP{\ref{IT_petalEST-hyperP}} By Lemma~\ref{LM_petal-into-petal}, there exists a petal $\Delta^\prime$ of $(\psi_t)$ such that $\varphi(\Delta)\subset\Delta^\prime$.
Being connected components of $\mathcal W^\circ$, petals are pairwise disjoint. Therefore, either ${\varphi(\Delta)\subset\Delta}$, or ${\varphi(\Delta)\cap\Delta=\emptyset}$. In the former case, by Corollary~\ref{CR_petals-iff-points}, ${\varphi(\sigma)=\sigma}$ in the sense of angular limits. Moreover, taking into account that ${\sigma\neq\tau}$, by Theorem~\ref{TH_FixP} we have that in this case, ${\varphi=\psi_{t_0}}$ for some ${t_0>0}$ and as a consequence $\varphi(\Delta)=\psi_{t_0}(\Delta)=\Delta$. Thus, if $\varphi(\Delta)\subset\Delta$, then alternative~\ref{IT_petalEST-hyperP(i)} in Theorem~\ref{petalEST}\,\ref{IT_petalEST-hyperP} holds.

Now suppose that $\varphi(\Delta)\cap\Delta=\emptyset$. If $\Delta'$ is a hyperbolic petal of~$(\psi_t)$, then according to Lemma~\ref{LM_correspondence of BRFPs}, alternative~\ref{IT_petalEST-hyperP(ii)} in Theorem~\ref{petalEST}\,\ref{IT_petalEST-hyperP} holds. If $\Delta'$ is a parabolic petal of~$(\psi_t)$, then again by Lemma~\ref{LM_correspondence of BRFPs}, ${\varphi(\sigma)=\tau}$ in the sense of angular limits. In this case, taking into account that ${\sigma\neq\tau}$ and that  $\varphi'(\tau)\neq\infty$, by \cite[Lemma~8.2]{Cowen-Pommerenke}, the angular derivative $\varphi'(\sigma)$ is~$\infty$, and thus, alternative~\ref{IT_petalEST-hyperP(iii)} holds. This completes the proof of Theorem~\ref{petalEST}\,\ref{IT_petalEST-hyperP}.

\StepP{\ref{IT_petalEST-paraP}} By  hypothesis, $\Delta$ is a parabolic petal for~$(\psi_t)$. Hence, statement~\ref{IT_petalEST-paraP(a)} is just Lemma~\ref{LM_parabolic-petal-into-itself}. Since one of the implications in statement~\ref{IT_petalEST-paraP(b)} is clear, see Section~\ref{petals}, it remains to prove that if ${\varphi(\Delta)=\Delta}$, then $\varphi$ is contained in~$(\psi_t)$. To this end, consider the canonical holomorphic model ${(S,H,(z\mapsto z+t)_{t\geq 0})}$ for the semigroup $(\psi_t)$. It is known, see e.g. \cite[Theorem~13.5.7]{BCD-Book}, that $\Pi:=H(\Delta)$ is a half-plane of the form ${\Pi=\{z:\Im z>a\}}$ or ${\Pi=\{z:\Im z<a\}}$. Note that ${(S,H,z\mapsto z+1)}$ is the canonical model for~$\psi_1$. Recall also that ${\varphi\in\Zen(\psi_1)}$. Therefore, according to \cite[Theorem~3.4]{CDG-Centralizer}, $\varphi={H^{-1}\circ g\circ H}$ for some ${g\in\U(S)}$ satisfying
\begin{equation}\label{EQ_comm-for-g}
  g(z+1)=g(z)+1 \quad\text{~for all $~{z\in S}$.}
\end{equation}
If $\varphi(\Delta)=\Delta$, then ${g(\Pi)=\Pi}$ and hence, taking into account~\eqref{EQ_comm-for-g}, we conclude that ${g(z)=z+c}$ for all ${z\in\Pi}$ and some constant~${c\in\Real}$. Clearly, this equality hold for all~${z\in S}$. Note that ${g\big(H(\UD)\big)}={H\big(\varphi(\UD)\big)}\subset{H(\UD)}$. Since $\psi_1$ is not an automorphism of~$\UD$, by \cite[Theorem~3.1\,(B)]{CDG-Centralizer}, we have ${c\ge0}$, and of course, ${c\neq0}$ because ${\varphi\neq\id_\UD}$. Therefore, ${\varphi=\psi_{t_0}}$ with ${t_0:=c>0}$.
\end{proof}

\begin{proof}[\proofof{Theorem \ref{Cor:petals}}] Let $(S,H,(z\mapsto z+t)_{t\geq 0})$ stand for the canonical holomorphic model of~$(\psi_t)$. As in the proof of Theorem~\ref{petalEST}, we see that~$\varphi$ is not a hyperbolic automorphism.

\StepP{\ref{IT_CorNew:if-hyperbolic-petal}}
The semigroup $(\psi_t)$ cannot be elliptic, because ${\psi_1\neq\id_\UD}$ commutes with the non-elliptic self-map~$\varphi$. Furthermore, $(\psi_t)$ cannot be hyperbolic because otherwise having ${\varphi\in\Zen(\psi_1)}$ and ${\psi_1\not\in\Aut}$  would imply that ${\varphi=\psi_{t_0}}$  for some ${t_0\ge0}$, see e.g. \cite[Remark~6.4]{CDG-Centralizer}; but this directly contradicts the hypothesis. Therefore, $(\psi_t)$ is parabolic, and by \cite[Corollary~4.1]{Cowen-comm} so is~$\varphi$. Moreover, by Lemma~\ref{LM-only-one-parabolic-petal}, $(\psi_t)$ can have at most one parabolic petal.

\StepP{\ref{IT_CorNew:petals-finite-seq}} Let $\Delta_*$ be the unique parabolic petal of~$(\psi_t)$.  It is known, see e.g. \cite[Theorem~13.5.7]{BCD-Book}, that $\Pi:=H(\Delta_*)$ is a half-plane of the form ${\Pi=\{z:\Im z>a\}}$ or ${\Pi=\{z:\Im z<a\}}$. Below we give a proof for the former alternative. The other case can be treated in a similar way.

 If $(\psi_t)$ is parabolic of zero hyperbolic step, then by \cite[Proposition~4.3]{CDG-Centralizer}, we have ${H\circ\varphi\circ H^{-1}}={(w\mapsto w+c)}$, where ${c\not\in[0,+\infty)}$ because $\varphi$ is not an element of~$(\psi_t)$. Note also that by \cite[Theorem~3.1\,(B)]{CDG-Centralizer}, ${c\not\in(-\infty,0)}$ because $\psi_1$ is not an automorphism. Moreover, ${\varphi(\Delta_*)\subset\Delta_*}$ by Lemma~\ref{LM_parabolic-petal-into-itself} and hence, ${\Pi+c\subset\Pi}$. As a consequence, ${\Im c>0}$. Since for each hyperbolic petal~$\Delta$, its image $H(\Delta)$ is a horizontal strip, see e.g. \cite[Theorem~13.5.5\,(2)]{BCD-Book}, it follows that there exists ${n\in\Natural}$ such that ${H(\Delta)+nc\subset\Pi}$ and, as a result, we have ${\varphi^{\circ n}\big(\Delta\big)\subset\Delta_*}$. Choose the smallest of such~$n$'s. To complete the proof of assertion~\ref{IT_CorNew:petals-finite-seq} for the case of zero hyperbolic case, it now suffices to appeal to Lemmas~\ref{LM_petal-into-petal} and~\ref{LM_correspondence of BRFPs}.

  Suppose $(\psi_t)$ is of positive hyperbolic step. We may suppose that~${S=\UH}$. The argument for the case ${S=-\UH}$ is similar. By \cite[Theorem~3.4]{CDG-Centralizer}, there exists $g\in\U(\UH)$ satisfying ${H\circ\varphi}={g\circ H}$ and such that ${g(w+1)=g(w)+1}$ for any ${w\in\UH}$.
 As in the previous case, $g$ cannot be of the form ${w\mapsto w+c}$ with $c\in\Real$. Therefore, see e.g. \cite[Lemma~9.3\,(II)]{CDG-Centralizer},  ${\exp\big(2\pi i g(\zeta)\big)}={f(e^{2\pi i \zeta})}$, where $f\in\Hol(\UD)$, ${f(0)=0}$, ${|f'(0)|<1}$. Since the iterates of such a self-map $f:\UD\to\UD$ converge to~$0$, see e.g. \cite[Theorem~3.1.13]{Abate2}, it follows that for any ${w\in\UH}$,  $\Im g^{\circ n}(w)\to+\infty$ as ${n\to+\infty}$. Therefore, any point~$z\in\UD$,
 $$
   \varphi^{\circ n}(z)=H^{-1}\big(g^{\circ n}(H(z))\big)\in H^{-1}(\Pi)=\Delta_*
 $$
 for $n\in\Natural$ large enough. As a consequence, for any hyperbolic petal~$\Delta$ there exists ${n\in\Natural}$ such that ${\varphi^{\circ n}(\Delta)\cap\Delta_*\neq\emptyset}$. Relying on the fact that petals are pairwise disjoint and applying Lemma~\ref{LM_petal-into-petal} with~$\varphi^{\circ n}$ on place of~$\varphi$, we therefore conclude that ${\varphi^{\circ n}(\Delta)\subset\Delta_*}$. Now we can complete the proof of~\ref{IT_CorNew:petals-finite-seq} in the same manner as in the previous case: pass to the smallest ${n\in\Natural}$ satisfying ${\varphi^{\circ n}(\Delta)\subset\Delta_*}$ and appeal to Lemmas~\ref{LM_petal-into-petal} and~\ref{LM_correspondence of BRFPs}.

\StepP{\ref{IT_CorNew:petals-inf-seq}} Suppose that $(\psi_t)$ has no parabolic petal and let~$\Delta$ be some hyperbolic petal of~$(\psi_t)$. Applying inductively Lemmas~\ref{LM_petal-into-petal}
and~\ref{LM_correspondence of BRFPs}, we conclude that there exists a sequence of hyperbolic petals $(\Delta_n)$ with $\alpha$-points~$\sigma_n$ such that ${\Delta_1=\Delta}$, ${\varphi(\Delta_n)\subset \Delta_{n+1}}$ and ${\varphi(\sigma_n)=\sigma_{n+1}}$ for all ${n\in\Natural}$. In order to see that $\Delta_n$'s are pairwise disjoint\footnote{An alternative way to prove this fact is to combine \cite[Proposition~5.2]{Filippo-puntosTAMS} with Lemma~\ref{LM_correspondence of BRFPs}.}, we suppose on the contrary that ${\Delta_{m+k}=\Delta_m}$ for some ${m,k\in\Natural}$. Then ${\varphi^{\circ k}(\Delta_m)\subset\Delta_m}$, and hence by Theorem~\ref{petalEST}\,\ref{IT_petalEST-hyperP} applied to $\Delta$ and $\varphi$ replaced by $\Delta_m$ and $\varphi^{\circ k}$ we would have that ${\varphi^{\circ k}=\psi_{t_0}}$ for some~${t_0>0}$. By Lemma~\ref{LM_uniqueness}, this would further imply that $\varphi$ itself is contained in~$(\psi_t)$, which contradicts the hypothesis of the theorem.
To complete the proof, it remains to notice that according to \cite[Theorem~1.4]{Filippo-puntosTAMS}, ${\psi_1'(\sigma_n)\le\psi_1'(\sigma_1)}$ for all ${n\in\Natural}$ and, as a consequence, the Cowen\,--\,Pommerenke inequality \cite[Theorem~4.1\,(iii)]{Cowen-Pommerenke}
$$
  \sum_{n=1}^{+\infty}\frac{|\tau-\sigma_n|^2}{\psi_1'(\sigma_n)-1}~\le 2\Re\Big(\tfrac{1}{\psi_1(0)}-1\Big)~<~+\,\infty
$$
implies that $|\tau-\sigma_n|\to0$ as ${n\to+\infty}$.
\end{proof}

\section{Petals and isogonality}\label{PetalosISO}
Given ${\varphi\in\U(\D)}$ non-elliptic  and ${(\psi_t)\not\subset\Aut}$ a continuous one-parameter semigroup in the disc such that $\psi_1\in\Zen(\varphi)\setminus\{\id_\UD\}$, in Theorem~\ref{petalEST}, it was shown that $\varphi$ maps petals of $(\psi_t)$ into petals of $(\psi_t)$. In fact, if such a petal $\Delta$ is parabolic, we always have ${\varphi(\Delta)\subset\Delta}$. However, this is no longer the case when $\Delta$ is hyperbolic. Apart from the same possibility, i.e. apart from the case ${\varphi(\Delta)\subset\Delta}$, and in general, we can have the other two following situations:
\begin{itemize}
	\item[]Situation A:
      {\slshape $\varphi(\Delta)\subset\Delta^\prime$ and $\Delta^\prime$ is a hyperbolic petal.}
	\item[]Situation B:
     {\slshape $\varphi(\Delta)\subset\Delta^\prime$ and $\Delta^\prime$ is a parabolic petal.}
\end{itemize}
In the following two results, we analyse when $\varphi$ maps $\Delta$ \textit{onto} $\Delta^\prime$ in both situations.

\begin{proposition}\label{PR_isogonal1}
	Assume we are in {\rm Situation A} and let $\sigma$ and $\sigma^\prime$ be the $\alpha$-points of  $\Delta$ and $\Delta^\prime$, respectively. Then the following conditions are equivalent:
\begin{equilist}
  \item\label{IT_isogonal-onto} $\varphi(\Delta)=\Delta'$;
  \item\label{IT_isogonal-isogonal} $\varphi$ is isogonal at~$\sigma$;
  \item\label{IT_isogonal-equal-spec-values} the spectral values of $(\psi_t)$ at $\sigma$ and $\sigma'$ coincide.
\end{equilist}
Moreover, if the above equivalent conditions hold, then $\varphi$ is affine with respect to $\psi_1$ and hence, ${(\psi_t)\subset\Zen(\varphi)}$.
\end{proposition}

\begin{proposition}\label{PR_isogonal2}
	In {\,\rm Situation B} we always have $\varphi(\Delta)\neq\Delta^\prime$.
\end{proposition}

The proof of these two propositions require some backgrounds which we expose thereafter. We begin by recall the definition of isogonality, which we extend to the case of infinite angular limit.
\begin{definition}
A univalent function $f:\UD\to\C$ is said to be \textsl{isogonal} or \textsl{semi-conformal} at a point~${\zeta\in\UC}$ if the following two angular limits exist:
\begin{equation*}
f(\zeta):=\anglim_{z\to\zeta} f(z)\in\Complex\cup\{\infty\} \quad\text{and}\quad\anglim_{z\to\zeta}\arg\frac{f_0(z)}{z-\zeta},
\end{equation*}
where $\arg$ is to be understood as a continuous map from $\C\setminus\{0\}$ to ${\Real/(2\pi\mathbb Z)}$ and
$$
 f_0(z):=\begin{cases}
 f(z)-f(\zeta)& \text{if~$f(\zeta)\neq\infty$},\\
 1/f(z)       & \text{if~$f(\zeta)=\infty$.}
 \end{cases}
$$

Further, a univalent function $\Psi:\Hr\to\C$ is said to be isogonal at $\omega\in\partial\Hr$ if the composition $f:=\Psi\circ H_\omega$, where  $H_\omega$  is the Cayley map of $\UD$ onto~$\Hr$ with ${H_\omega(0)=0}$, ${H_\omega(1)=\omega}$, is isogonal at ${\zeta=1}$.
\end{definition}

\begin{remark}\label{RM_iso_UD}
If $f\in\U(\UD)$ and $\zeta$ is a contact point of~$f$, then an elementary argument shows that the isogonality condition requires that the image of the radial segment $[0,\zeta]$ under~$f$ approaches the point~${f(\zeta)\in\UC}$ orthogonally to~$\UC$. Therefore, in this case, the isogonality at~$\zeta$ is equivalent to
\begin{equation}\label{EQ_iso_UD}
\anglim_{z\to\zeta}\Arg\frac{1-\overline{f(\zeta)}f(z)}{1-\overline\zeta z}=0,
\end{equation}
where $\Arg w$ stands for the unique value of the argument of~$w\neq0$ contained in~${(-\pi,\pi]}$.
Passing this remark to the right half-plane with the help of the Cayley map, we can say that a univalent self-map ${T\in \U(\Hr)}$ with a boundary fixed point at~$0$ is isogonal at~$0$ if and only if
\begin{equation}\label{EQ_iso_UH}
\anglim_{w\to 0}\Arg\frac{T(w)}{w}=0.
\end{equation}
\end{remark}
\begin{definition}\label{DF_premodel}
Let $\sigma\in\UC$ be a repelling fixed point of a continuous one-parameter semigroup $(\psi_t)$ with associated infinitesimal generator~$G$. The triple $(\Hr,\Psi,Q_t)$ is called a pre-model for $(\psi_t)$ at~$\sigma$ if the following conditions are met:
\begin{itemize}
\item[(i)] for each $t\ge0$, $Q_t$ is the automorphism of~$\Hr$ given by $Q_t(z):=e^{\lambda t}z$, where $\lambda:=G'(\sigma)$;
\item[(ii)] the map $\Psi:\Hr\to\UD$ is holomorphic and injective, $\anglim_{w\to0}\Psi(w)=\sigma$, and $\psi$ is isogonal at~$0$, i.e.
    \begin{equation}\label{EQ_isogonality}
      \anglim_{w\to0}\Arg\frac{1-\overline{\sigma}\Psi(w)}w=0;
    \end{equation}
\item[(iii)] $\Psi\circ Q_t=\psi_t\circ\Psi$ for all~$t\ge0$.
\end{itemize}
\end{definition}

\begin{remark}\label{RM_pre-model}
It is known \cite[Theorem~3.10]{Bracci_et_al2019} that every continuous one-parameter semigroup, at each repelling fixed point~$\sigma$, admits a pre-model unique up to the transformation $\Psi(w) \mapsto \Psi(cw)$, where $c$ is an arbitrary positive constant. Moreover, $\Psi(\Hr)$ coincides with the hyperbolic petal $\Delta(\sigma)$ associated with $\sigma$.
The map~$\Psi$ can be expressed via the Koenigs function~$h$ of~$(\psi_t)$. Namely, if the strip~$h(\Delta(\sigma))$ is $\mathbb S(a,b)={\{w:a<\Im w<b\}}$, then the intertwining map $\Psi$ in the pre-model for~$(\psi_t)$ at~$\sigma$ is given by
$$
\Psi(w):=h^{-1}\big(\tfrac{b-a}{2\pi}\log w+\tfrac{b+a}{2}i+s\big),\quad w\in\Hr,
$$
where $s$ is an arbitrary real constant.
\end{remark}

As the following result shows, it is also possible to talk about pre-models associated with parabolic petals.
\begin{theorem} \cite[Proposition 13.4.10 and its proof]{BCD-Book} Let $(\psi_t)$ be a parabolic semigroup in the unit disc with Denjoy\,--\,Wolff point $\tau\in\partial\D$ and let $\Delta$ be a parabolic petal of $(\psi_t)$. Then, there exist $\Psi\in\U(\Hr,\D)$ with $\Psi(\Hr)=\Delta$ and $\angle \lim_{w\to 0}\Psi(w)=\tau$ and a parabolic group $(Q_t)$ in $\Hr$ with Denjoy\,--\,Wolff point at~$0$ such that, for  all $t\geq 0$,
$$
\Psi\circ Q_t=\psi_t\circ\Psi.
$$	
\end{theorem}

\begin{definition}
In the notation of the above theorem,  the triple $(\Hr,\Psi,Q_t)$ is called a \textsl{pre-model} for $(\psi_t)$ associated with the parabolic petal $\Delta$.
\end{definition}

\begin{lemma}\label{LM_isogonal}
Let $T$ be a univalent self-map of~$\Hr$. If $\,T(0)=0\,$ in the sense of angular limits and if $T$ is isogonal at~$0$, then for any ${\beta>0}$,
\begin{equation}\label{EQ_LM_isogonal}
\lim_{x\to0^+}\frac{T(\beta x)}{T(x)}~=~\beta.
\end{equation}
\end{lemma}
\begin{proof}
Since $T$ is isogonal at~$0$ with ${T(0)=0}$, the univalent function $T\big(\tfrac{1-z}{1+z}\big)$, ${z\in\UD}$, satisfies at ${\zeta=1}$ the Visser\,--\,Ostrowski condition; see e.g. \cite[Proposition~4.11 on p.\,81]{Pommerenke:BB}. It follows\footnote{In~\cite{Wuerzburg} one can find another proof of this fact (not assuming global univalence of~$T$), see \cite[Lemma~8.1]{Wuerzburg} applied to~${z\mapsto1/T(1/z)}$, as well some closely related results connecting asymptotic behaviour of the hyperbolic derivative to isogonality and to existence of finite angular derivative.} that
 \begin{equation*}
\lim_{x\to0^+}\frac{x T'(x)}{T(x)}~=~1.
 \end{equation*}
As a consequence,
$$
\log\frac{T(\beta x)}{T(x)}=\int\limits_1^\beta\,\frac{tx\, T'(tx)}{T(tx)}\,\frac{\di t}{t}~\longrightarrow~\int\limits_1^\beta\,\frac{\di t}{t}~=~\log\beta
$$
as $x\to0^+$, as desired.
\end{proof}

   Recall that a sequence $(z_n)\subset\UD$ converging to some point ${\sigma\in\UC}$ is said to converge to $\sigma$ \hbox{\textsl{non-tangentially}} if the limit set ${\mathrm{Slope}[(z_n),n\to+\infty]}$ of ${\Arg(1-\overline\sigma z_n)}$ as ${n\to+\infty}$ is compactly contained in the open interval $(-\pi/2,\pi/2)$. Further, we say that $(z_n)$ converges to~$\sigma$ \textsl{tangentially}, if ${\mathrm{Slope}[(z_n),n\to+\infty]}\subset{\{-\pi/2,\pi/2\}}$. Clearly, $(z_n)$ converges to~$\sigma$ non-tangentially if and only if it has no subsequence converging tangentially.

\begin{lemma}\label{LM_isogonal-tangential}
Suppose $f\in\U(\UD)$ is isogonal at some point ${\zeta\in\UC}$ and that ${f(\zeta)\in\UC}$. If ${(z_n)\subset\UD}$ converges to~$\zeta$ tangentially, and if $(w_n)$, given by ${w_n:=f(z_n)}$ for all ${n\in\Natural}$,  converges to $f(\zeta)$, then the convergence of $(w_n)$ is also tangential.
\end{lemma}
\begin{proof}
Composing $f$ with suitable conformal mappings we may replace $\UD$ with~$\Hr$ and suppose that ${f(\zeta)=\zeta=0}$.
Fix some ${\eta\in(0,\pi/2)}$ and set $\theta:=\tfrac12(\eta+\pi/2)$. Since $(z_n)\subset\Hr$ converges to~$0$ tangentially,  there exists ${r>0}$ such that $A_{\theta,r}:=\{z:|\Arg z|<\theta,~|z|<r\}$ does not contain any point of~$(z_n)$. The isogonality of~$f$ at~$0$ implies, see e.g. \cite[Proposition~4.10 on p.\,81]{Pommerenke:BB}, that there exists ${\rho>0}$ such that $A_{\eta,\rho}\subset {f\big(A_{\theta,r}\big)}$. Since $f$ is univalent by hypothesis, it follows that $A_{\eta,\rho}$ does not contain any point of~$(w_n)$. Taking into account that ${\eta\in(0,\pi/2)}$ in this argument is arbitrary, we conclude that if $(w_n)$ converges to~$0$, then the convergence must be tangential.
\end{proof}

\begin{lemma}\label{LM_isog-compositions}
Let $D_1,D_2,D_3\in\{\UD,\Hr\}$ and let $f_3:=f_2\circ f_1$, where ${f_1:D_1\to D_2}$ and ${f_2:D_2\to D_3}$ are two (holomorphic) univalent functions. The following two statements hold.
\begin{statlist}
  \item\label{IT_iso-comp-f2-and-f1}
  If $f_1$ is isogonal at some $\zeta_1\in\partial D_1$ with $\zeta_2:=f_1(\zeta_1)\in\partial D_2$ and if $f_2$ is isogonal at~$\zeta_2$ with $\zeta_3:=f_2(\zeta_2)\in\partial D_3$, then $f_3$ is isogonal at~$\zeta_1$ with $f_3(\zeta_1)=\zeta_3$.
  \smallskip
  \item\label{IT_iso-comp-f3} Conversely, if $f_3$ is isogonal at some $\zeta_1\in\partial D_1$ and if $\zeta_3:={f_3(\zeta_1)\in\partial D_3}$, then: {\rm (i)}~$f_1$ is isogonal at~$\zeta_1$, {\rm (ii)}~$\zeta_2:=f_1(\zeta_1)$ belongs to~$\partial D_2$, {\rm (iii)}~$f_2$ is isogonal at~$\zeta_2$, and {\rm (iv)}~${f_2(\zeta_2)=\zeta_3}$.
\end{statlist}
\end{lemma}

\begin{proof} First of all, without loss of generality, we may suppose that ${D_1=D_2=D_3=\UD}$.

\StepP{\ref{IT_iso-comp-f2-and-f1}}
 Taking into account Remark~\ref{RM_iso_UD}, we have to show that for any sequence ${(z_n)\subset\UD}$ converging to~$\zeta_1$ non-tangentially, $f_3(z_n)\to\zeta_3$ and $\Arg\big((1-\overline{\zeta_3}f_3(z_n))/(1-\overline{\zeta_1}z_n)\big)\to0$. This follows readily from the definition of isogonality and Remark~\ref{RM_iso_UD} applied for ${f:=f_1}$ and for ${f:=f_2}$, if we additionally notice that by~\eqref{EQ_iso_UD} with~${f:=f_1}$ and ${\zeta:=\zeta_1}$, the sequence formed by the points $w_n:=f_1(z_n)$, ${n\in\Natural}$, converges to~$\zeta_2$ non-tangentially.

\StepP{\ref{IT_iso-comp-f3}} Recalling that $f_3(\D)\subset\D$, the necessary and sufficient geometric condition for isogonality  due to Ostrowski, see e.g. \cite[Theorem~11.6 on p.\,254]{Pommerenke:BB}, applied to~$f_3$ at $\zeta_1$ takes the following form: there exist ${t_1>0}$ and a curve ${C:[0,t_1]\to\UD\cup\{\zeta_1\}}$ satisfying the following three conditions:
\begin{equilist}
  \item $C(0)=\zeta_1$ and $C(t)\in\UD$ for all ${t\in(0,t_1]}$;
  \item $f_3\big(C(t)\big)=w_0+\mu t~$ for some ${w_0\in\C}$, some ${\mu\in\UC}$, and all ${t\in(0,t_1]}$;
  \item for any $\varepsilon>0$, there exists $\rho=\rho(\varepsilon)>0$ such that
      $$
        \big\{w:\,|w-w_0|<\rho,~\Re\big(\overline\mu(w-w_0)\big)>\varepsilon^2|w-w_0|\big\}\subset f_3(\UD).
      $$
\end{equilist}
In particular, $w_0$ is an asymptotic value of~$f_3$ at~$\zeta_1$. By the Lindel\"of Theorem, see e.g. \cite[Theorem~9.3 on p.\,268]{Pombook75}, it follows that $w_0$ coincides with the angular limit of~$f_3$ at~$\zeta_1$, i.e. ${w_0=\zeta_3}$.

Applying the No-Koebe-Arc Theorem to the function~$f_2$, it is not difficult to show, see e.g.~\cite[Corollary~9.2 on p.\,267]{Pombook75}, that $C_1(t):=f_1\big(C(t)\big)$ has a limit~${\zeta_2\in\UC}$ as ${t\to0^+}$. Moreover, ${f_2(z)\to w_0=\zeta_3}$ as ${\Gamma\ni z\to\zeta_2}$, where $\Gamma:=C_1\big((0,t_1)\big)$. Again by the Lindel\"of Theorem,  it follows that, for $k=1,2$, the function $f_k$ has angular limit at~$\zeta_k$ equal to~$\zeta_{k+1}$.

Taking into account that $w_0=\zeta_3\in\UC$ and that $f_3(\UD)\subset f_2(\UD)\subset\UD$ and applying again the above characterization of isogonality but now to~$f_2$ at $\zeta_2$ with $C$ replaced by~$C_1$, we see that $f_2$ is isogonal at~$\zeta_2$.

It remains to see that $f_1$ is isogonal at~$\zeta_1$. Suppose the contrary. Then, according to Remark~\ref{RM_iso_UD}, there exists a sequence ${(z_n)\subset\UD}$ converging to~$\zeta_1$ non-tangentially such that
$$
\Arg\big(1-\overline{\zeta_1} z_n\big)\to\theta_1\quad\text{and}\quad
 \Arg\big(1-\overline{\zeta_2} f_1(z_n)\big)\to\theta_2\quad \text{as~$~n\to+\infty$}
$$
for some $\theta_{1}\in(-\pi/2,\pi/2)$ and some $\theta_2\in[-\pi/2,\pi/2]$, with ${\theta_2\neq\theta_1}$.
Since $f_3$ is isogonal at~$\zeta_1$ by  hypothesis, we have that $f_3(z_n)={f_2\big(f_1(z_n)\big)}$ converges to~$\zeta_3$ non-tangentially. Therefore, applying Lemma~\ref{LM_isogonal-tangential} with ${f:=f_2}$ and with~$z_n$ replaced by~${\tilde z_n:=f_1(z_n)}$, we see that $(\tilde z_n)$ cannot converge to~$\zeta_2$ tangentially. Hence, $\theta_2\in{(-\pi/2,\pi/2)}$ and as a consequence, we can use the isogonality of~$f_2$ at~$\zeta_2$ to see that $\Arg\big(1-\overline{\zeta_3} f_2\big(f_1(z_n)\big)\big)\to\theta_2$ as~$n\to+\infty$. This however contradicts the fact that by the isogonality of~$f_3$ at~$\zeta_1$,
$$
 \Arg\big(1-\overline{\zeta_3} f_2\big(f_1(z_n)\big)\big)~=~\Arg\big(1-\overline{\zeta_3} f_3(z_n)\big)~\to~\theta_1\neq\theta_2\quad \text{as~$~n\to+\infty$.} \eqno\qedhere
$$
\end{proof}

\begin{proof}[\proofof{Proposition~\ref{PR_isogonal1}}] Let ${(\UH,\Psi,Q_t)}$ and ${(\UH,\widetilde\Psi,\widetilde Q_t)}$ be pre-models associated with the petals $\Delta$ and $\Delta'$, respectively. Further, denote ${\Psi_1=\varphi\circ \Psi}$. Then we have
$$
\psi_1\circ\Psi_1 = \varphi \circ \psi_1 \circ \Psi =\varphi \circ \Psi\circ Q_1 = \Psi_1\circ Q_1.
$$
It follows that $\psi_1|_{\varphi(\Delta)}={\Psi_1\circ Q_1\circ \Psi_1^{-1}}$. At the same time, $\psi_1|_{\Delta'}={\widetilde\Psi\circ \widetilde Q_1\circ \widetilde\Psi^{-1}}$. Therefore, the three univalent self-maps of $\UH$: ${Q_1(w)=\psi_1'(\sigma)w}$, ${\widetilde Q_1(w)=\psi_1'(\sigma')w}$ for all ${w\in\UH}$, and $T:={\widetilde\Psi^{-1}\circ\Psi_1}$,  satisfy the identity
\begin{equation}\label{EQ_identity-for-T}
   T\circ Q_1 \,=\, \widetilde Q_1\circ T.
\end{equation}

\StepG{Step 1: \ref{IT_isogonal-onto} implies \ref{IT_isogonal-isogonal} and \ref{IT_isogonal-equal-spec-values}.} If $\varphi(\Delta)=\Delta'$, then ${\Psi_1(\UH)=\Delta'}$ and hence,  $T$ is an automorphism of~$\Hr$. Since conjugation does not change the multipliers of the fixed points and since the numbers $\psi_1'(\sigma)$ and $\psi_1'(\sigma')$ are both greater than~$1$, identity~\eqref{EQ_identity-for-T} implies in this case that~${T(0)=0}$, ${T(\infty)=\infty}$, and ${\psi_1'(\sigma)=\psi_1'(\sigma')}$. The latter equality means that~\ref{IT_isogonal-equal-spec-values} holds. The former two equalities imply that for a suitable constant ${\alpha>0}$, $T(w)=\alpha w$ and hence
\begin{equation}\label{EQ_varphi-Psi-Psi}
\varphi\big(\Psi(w)\big)=\widetilde\Psi(\alpha w)\quad\text{for all~$w\in\UH$.}
\end{equation}
To prove~\ref{IT_isogonal-isogonal}, i.e. isogonality of~$\varphi$ at~$b$, it is now sufficient to apply Lemma~\ref{LM_isog-compositions}\,\ref{IT_iso-comp-f3} for ${f_1:=\Psi}$, ${f_2:=\varphi}$, and ${f_3(w):=\widetilde\Psi(\alpha w)}$, ${w\in\Hr}$.

\StepG{Step~2: \ref{IT_isogonal-onto} implies that $\varphi$ is affine with respect to $\psi_1$.} Let $H$ stand for the Koenigs function of $(\psi_t)$. Then ${S:=H(\Delta)}$ and ${S':=H(\Delta')}$ are two horizontal strips (see \cite[Theorem~13.5.5]{BCD-Book}). Write $g:={H\circ \varphi\circ H^{-1}}$. Since ${\varphi\circ\psi_1}={\psi_1\circ\varphi}$, we have ${g(w+1)}={g(w)+1}$ for all ${w\in\Delta}$. Moreover, by the hypothesis, ${\varphi(\Delta)=\Delta'}$ and hence ${g(S)=S'}$. Taking into account that $g$ is univalent in~$S$, it follows that $g$ is an affine map, i.e. that $\varphi$ is affine w.r.t.~$\psi_1$.

\StepG{Step~3: \ref{IT_isogonal-isogonal} implies \ref{IT_isogonal-equal-spec-values}.} By the hypothesis, $\varphi$ is isogonal at~$\sigma$. Then, by Lemma~\ref{LM_isog-compositions}\,\ref{IT_iso-comp-f2-and-f1}, $\Psi_1$ is isogonal at~$0$, with $\Psi_1(0)=\varphi(\sigma)=\sigma'\in\partial\UD$. Since $\widetilde\Psi\circ T=\Psi_1$, by Lemma~\ref{LM_isog-compositions}\,\ref{IT_iso-comp-f3}, the function $T$ is isogonal at~$0$ and $\widetilde\Psi$ is isogonal at ${T(0)\in\partial\Hr}$, with $\widetilde\Psi(T(0))=\Psi_1(0)=\sigma'$. By the definition of a pre-model,  $\widetilde\Psi$ is also isogonal at~$0$ and $\widetilde\Psi=\sigma'$. Since by the same definition $\widetilde\Psi$ is univalent, it follows that ${T(0)=0}$.

Now, combining Lemma~\ref{LM_isogonal} and identity~\eqref{EQ_identity-for-T}, we see that ${\psi_1'(\sigma)=\psi_1'(\sigma')}$, i.e.~\ref{IT_isogonal-equal-spec-values} holds.

\StepG{Step~4: \ref{IT_isogonal-equal-spec-values} implies \ref{IT_isogonal-onto}.}  We use the same notation as in Step~2. Since ${\varphi(\Delta)\subset\Delta'}$, we have ${g(S)\subset S'}$. In accordance with \cite[Theorem~2.6]{Analytic-flows}, condition~\ref{IT_isogonal-equal-spec-values} implies that the strips ${S:=H(\Delta)}$ and ${S':=H(\Delta')}$ are of the same width. Therefore, for a suitable $a\in\Real$, $g_0:={g-ia}$ is a self-map of~$S$ commuting with $w\mapsto{w+1}$. The the result of Heins~\cite[Lemma~2.1]{Heins}, this means that $g_0$ maps $S$ \textit{onto} itself and hence, ${\varphi(\Delta)=\Delta'}$.
\end{proof}

\begin{proof}[\proofof{Proposition~\ref{PR_isogonal2}}]
Let ${(\UH,\Psi,Q_t)}$ be a pre-model associated with the hyperbolic petal $\Delta$. Hence, for some $\lambda>0$, we have ${Q_t(w)=e^{\lambda}w}$ for all ${w\in\Hr}$ and all ${t\geq0}$. Likewise, let ${(\UH,\widetilde\Psi,\widetilde Q_t)}$ be a pre-model associated with the parabolic petal~$\Delta^\prime$. Then $(\widetilde Q_t)$ is a parabolic group in $\Hr$ with Denjoy\,--\,Wolff point at~$0$, i.e.  $\widetilde Q_t(w)=\frac{w}{1+ictw}$ for all $w\in\Hr$,  all ${t\in\Real}$ and some constant ${c\in\R^*}$. Denote ${\Psi_1:=\varphi\circ \Psi}$. Arguing exactly as in the beginning of the proof of Proposition~\ref{PR_isogonal1}, we obtain that, for all $w\in\Delta^\prime$,
$$
{\Psi_1\circ Q_1\circ \Psi_1^{-1}}(w)={\widetilde\Psi\circ \widetilde Q_1\circ \widetilde\Psi^{-1}}(w).
$$
Now, consider $T:={\widetilde\Psi^{-1}\circ\Psi_1}$. The above identity implies that that $T\circ Q_t \,=\, \widetilde Q_t\circ T$, i.e.
$$
T\left(e^{\lambda t}w\right)=\frac{T(w)}{1+ictT(w)},\ w\in\Hr,\ t\geq0.
$$

Assume that $\varphi(\Delta)=\Delta^\prime$. Then $T$ is an automorphism of $\Hr$ thus a  linear fractional map. In particular, it follows that the above equality extends to~$\C\cup\{\infty\}$. Setting ${w:=0}$ and $w:=\infty$,
$$
T(0)=\frac{T(0)}{1+ictT(0)}\quad\text{and}\quad T(\infty)=\frac{T(\infty)}{1+ictT(\infty)}\quad\text{for all~$~t\ge0$}.
$$
It immediately follows that $T(0)=T(\infty)=0$, which is impossible because the linear-fractional map $T$ is not constant. This contradiction show that ${\varphi(\Delta)\neq\Delta'}$.
\end{proof}

\section{Proof of Theorem \ref{Th_main}}\label{S_proof-of-the-main-Theorem}
 Since $\varphi$ is non-elliptic and since $\psi_1\neq\id_\UD$ commutes with~$\varphi$, it is elementary to see that the self-map $\psi_1$, and hence the whole semigroup $(\psi_t)$ is also non-elliptic.
Let $(S,H,(z\mapsto z+t)_{t\geq0})$ stand for the canonical holomorphic model of~$(\psi_t)$.

Sufficiency of conditions~\ref{IT_main:affine}, \ref{IT_main:isogonal}, and~\ref{IT_main:boundary-fixed-pt} follows directly from Proposition~\ref{PR_simple}, Proposition~\ref{PR_isogonal1}, and Theorem~\ref{TH_FixP}, respectively.
It remains to prove the result for conditions~\ref{IT_main:hyperbolic-or-parabolic-zero}, \ref{IT_main:two-values-of-t}, and~\ref{IT_main:unrestr_limit}.

Suppose \ref{IT_main:hyperbolic-or-parabolic-zero} holds, i.e.,
$\psi_{1}$ is hyperbolic or parabolic of zero hyperbolic step. Then by \cite[Proposition~4.3]{CDG-Centralizer},  there exists a constant $c\in \C$ such that $H\circ \varphi=H+c$, and the desired conclusion that ${(\psi_t)\subset\Zen(\varphi)}$ follows immediately by Proposition~\ref{PR_simple}.

Suppose \ref{IT_main:two-values-of-t} holds. Note that $(S,H,z\mapsto z+1)$ is a holomorphic model for~$\psi_1$. By the hypothesis, ${\varphi\in\Zen(\psi_1)}$. Therefore, by \cite[Theorem~3.3]{CDG-Centralizer}, the function $g:=H^{-1}\circ\varphi\circ H$ extends holomorphically from $H(\UD)$ to the whole domain~$S$ and satisfies
\begin{equation}\label{EQ_periodic_1}
f(w+1)=f(w):=g(w)-w\quad\text{for all~$~w\in S$.}
\end{equation}
Moreover, by~\ref{IT_main:two-values-of-t}, $\psi_r\circ\varphi=\varphi\circ\psi_r$. It follows that
\begin{equation}\label{EQ_periodic_r}
f(w+r)=f(w)\qquad\text{for all~$~w\in S$.}
\end{equation}
Fix any $w_0\in S$. Then $w_0+\Real\subset S$. Recall that $r\in\Real\setminus\Q$. Hence, ${\{k+jr:k,j\in\mathbb Z\}}$ is dense in~$\Real$. By~\eqref{EQ_periodic_1} and~\eqref{EQ_periodic_r} it follows that $f$ is constant on ${w_0+\Real}$. Therefore, $f$ is constant in~$S$, i.e. $\varphi$ is affine with respect to $\psi_1$. Hence, again  by Proposition~\ref{PR_simple}, ${(\psi_t)\subset\Zen(\varphi)}$.

In the last part of the proof we may suppose that $(\psi_t)$ is parabolic of positive hyperbolic step, since otherwise, as we have seen above, the desired conclusion ${(\psi_t)\subset\Zen(\varphi)}$ holds automatically. In this situation, applying Theorem~\ref{TH_affine} with $\psi$ and~$\varphi$ replaced by $\varphi$ and~$\psi_1$, respectively, we conclude that condition~\ref{IT_main:unrestr_limit} implies that $\varphi$ is affine w.r.t.~$\psi_1$. It remains to appeal to Proposition~\ref{PR_simple} for third time and we are done.
\qed

\section{ Commuting one-parameter semigroups}\label{Sec:commutingsemigroups}
In this final section, we apply our results to the case when $\varphi $ is embeddable. Namely, assume we have two continuous one-parameter semigroups $(\varphi_{t})$ and $(\psi_{t})$ such that $\varphi_{1}\circ \psi_{1}=\psi_{1}\circ \varphi_{1}$. The question we address is whether $(\varphi_{t})$ and $(\psi_{t})$ commute, i.e. whether $\varphi_{s}\circ \psi_{t}=\psi_{t}\circ \varphi_{s}$ for all $t,s \geq 0$. As Example~\ref{Mainexample} shows, the answer is negative in general.  In fact, the semigroups provided in that example are parabolic and one of them is of zero hyperbolic step, while the other is of positive hyperbolic step. This behaviour is not a coincidence. In fact, if $\varphi_{1}$ is hyperbolic, then so is $\psi_{1}$ (see \cite[Corollary~4.1]{Cowen-comm}), and we can apply twice Theorem~\ref{Th_main}\,\ref{IT_main:hyperbolic-or-parabolic-zero} to get a positive answer to the commutativity between the two semigroup. The same happens if both $\varphi_{1}$ and $\psi_{1}$ are parabolic of zero hyperbolic step. Indeed, these two results can also be obtained from  \cite[Theorems~4 and~5(i)]{conReich}. Being negative the answer when at least one of the semigroups is parabolic of positive hyperbolic step, we could ask ourselves about some extra condition to assure commutativity. As a consequence of Theorem~\ref{TH_affine} and Theorem~\ref{Th_main}\,\ref{IT_main:affine} we have:

\begin{corollary}\label{Cor:semigroups} Let $(\varphi_{t})$ and $(\psi_{t})$ be two non-elliptic semigroups such $\varphi_{1}$ and $\psi_{1 }$ commute. Assume that there exists the unrestricted limit
\begin{equation}\label{Eq:semigroups}
\lim_{z\to \tau}\frac{\varphi_{1}(z)-z}{\psi_{1}(z)-z}.
\end{equation}
Then the semigroups $(\varphi_{t})$ and $(\psi_{t})$ commute.
\end{corollary}

This corollary is a generalization of \cite[Theorem~5(ii)]{conReich}, where the commutativity of $(\varphi_{t})$ and $(\psi_{t})$ is assured under the following three hypothesis: (i)~$\varphi_{1}$ and $\psi_{1 }$ commute; (ii)~for all~${t\geq 0}$, $\varphi_{t}$~and $\psi_{t}$ have unrestricted limits at the common Denjoy\,--\,Wolff point~$\tau$;  and also (iii)~the unrestricted limits $\lim_{z\to \tau}F''(z)$ and $\lim_{z\to \tau}G''(z)$, where $F$ and $G$ stand for the infinitesimal generators of $(\varphi_{t})$ and $(\psi_{t})$, respectively, exist and are both different from zero. But, denoting $A:=\lim_{z\to \tau}F''(z)$ and $B:=\lim_{z\to \tau}G''(z)$, by \cite[Remark 3]{conReich}, we have that
$$
\frac{A}{2}=\lim_ {z\to \tau} \frac{\varphi_1(z)-\tau}{(z-\tau)^2}, \quad \frac{B}{2}=\lim_ {z\to \tau} \frac{\psi_1(z)-\tau}{(z-\tau)^2}.
$$
Therefore, under those three hypothesis, the unrestricted limit in \eqref{Eq:semigroups} does exist and, in fact, it is equal to~${A/B}$.

\section{Appendix: the elliptic case}
\label{Sec:elliptic case}
In \cite[Corollary~4.9]{Cowen-comm}, Cowen proved that if two holomorphic self-maps $\varphi,\psi\in\Hol(\UD)\setminus\Aut$ commute and $\psi$ is univalent and either elliptic or hyperbolic, then $\varphi$ is also univalent. Furthermore, under the same assumptions, if another self-map $\widetilde{\psi}\in\Hol(\UD)\setminus\Aut$ commutes with $\psi$, then $\widetilde{\psi}$ commutes  with $\varphi$ as well. It follows that, in the elliptic case, Problem~\ref{theproblem} has a positive answer, at least if automorphisms of~$\UD$ are excluded from consideration. Below we provide an elementary proof of this fact thus providing a positive answer to Problem~\ref{theproblem} when $\varphi$ is elliptic (except for the case of an automorphism satisfying ${\varphi^{\circ n}=\id_\UD}$ for some ${n\in\Natural}$).

\begin{proposition}\label{PR_elliptic-case}
Let $\varphi\in \Hol(\D)$ and let $(\psi_{t})$ be an elliptic continuous one-parameter semigroup in $\UD$. Suppose that $\psi_1$ is not an automorphism and that ${\varphi\circ\psi_{1}}={\psi_{1}\circ \varphi}$. Then $\varphi$ is univalent and ${\psi_{t}\in \Zen(\varphi)}$ for all ${t>0}$.
\end{proposition}
\begin{proof}
According to the hypothesis of the proposition and Subsections \ref{Sec:modelos} and  \ref{SS_one-param-semigr}, $(\psi_t)$ admits a holomorphic model of the form ${\big(\C,H,(z\mapsto e^{\lambda t}z)_{t\ge0}\big)}$ with ${\Re\lambda<0}$.
Write $\Omega=H(\D)$. Consider $g:=H\circ \varphi \circ H^{-1}:\Omega \to \Omega$. Since $\psi^{\circ n}$ commutes with $\varphi$ for all ${n\in\N}$, we have
$$
g(e^{-\lambda n}w)=e^{-\lambda n}g(w), \quad \textrm{ for all }~w\in \Omega~\text{ and all }~n\in\N.
$$
In particular, ${g(0)=0}$ and it follows that the function $f(w):={g(w)/w}$, ${w\in \Omega\setminus \{0\}}$, has a holomorphic extension to ${w=0}$. Moreover,
$$
f(e^{-\lambda n}w)=\frac{e^{-\lambda n}g(w)}{e^{-\lambda n}w}=f(w) \quad \textrm{ for all }~w\in \Omega~\text{ and all }~n\in\N.
$$
Passing to the limit as~${n\to+\infty}$, we have that $f(w)=f(0)=:c$ for all ${w\in\Omega}$, i.e. ${g(w)=cw}$ in~$\Omega$. Therefore, $\varphi$ is univalent and commutes with $\psi_{t}$ for all ${t>0}$.
\end{proof}

The situation is different for automorphisms. More precisely, suppose that $\psi_1$ is an elliptic automorphism with the fixed point ${\tau\in\UD}$, and hence so are all $\psi_t$'s. Write $\mu:=\psi_1'(\tau)$.  It is not difficult to show, see e.g. \cite[Proposition~2.6.10]{Abate2}, that:
\begin{romlist}
\item\label{IT_irratioan-rotation} if the multiplier $\mu$ is not a root of the unity and if ${\varphi\in\Hol(\UD)}$ commutes with~$\psi_1$, then $\varphi$ is a linear-fractional map with the same fixed points as $\psi_1$; hence, in this case, $\varphi$ commutes with~$\psi_t$ for all~${t>0}$.
\item\label{IT_rational-rotation} However, if $\mu^{n}=1$ for some ${n\in\Natural}$, then ${\varphi\in\Hol(\UD)}$ commutes with~$\psi_1$ if and only if $\varphi$ is of the form $\varphi={\gamma^{-1}\circ \big(z\mapsto z \phi(z^n)\big) \circ \gamma}$, where $\phi\in\Hol(\UD)$ and $\gamma(z):={(z-\tau)/(1-\overline{\tau}z)}$.
\end{romlist}
An argument similar to that given in the proof of Proposition~\ref{PR_elliptic-case} shows that in case~\ref{IT_rational-rotation}, $\varphi$ commutes with all~$\psi_t$'s if and only if the function~$\phi$ is constant, which means that there are many $\varphi$'s that commute with $\psi_1$ but not with  the whole semigroup~$(\psi_t)$.

\end{document}